\documentclass[a4paper,reqno]{amsart}
\usepackage[utf8]{inputenc}
\usepackage{amssymb, slashed}
\usepackage{enumitem}
\usepackage{mathrsfs}
\usepackage{mathtools}
\usepackage{amsmath}
\usepackage [dvipsnames] { xcolor }
\usepackage[colorlinks=true]{hyperref}

\usepackage{bbm}

\hypersetup{linkcolor=cyan,urlcolor=green,citecolor=green}

\setlist[enumerate]{label=\emph{(\roman*)}}

\usepackage{environ}

\newtheorem{theorem}{Theorem}[section]

\newtheorem{lemma}[theorem]{Lemma}
\newtheorem{proposition}[theorem]{Proposition}

\theoremstyle{definition}
\newtheorem{definition}[theorem]{Definition}
\newtheorem{remark}[theorem]{Remark}
\numberwithin{equation}{section}

\def \R {{\mathbb{R} }}
\def \d {{\rm{d}}}
\def \pt{\partial_{t}}
\def \pr{\partial_{r}}

\def\Deltas{\slashed{\Delta}}

\def\cfrak{\mathfrak{c}}

\def\MM{\mathcal{M}}

\def\phiinit{\phi_{init}}

\def\phih{\hat{\phi}}

\def\psiinit{\psi_{init}}
\def\phihinit{\phih_{init}}

\def\psih{\hat{\psi}}
\def\psihinit{\psih_{init}}

\def\ep{\epsilon}

\definecolor{mycolor}{RGB}{139,170,65}

\newcounter{mnotecount}[section]

\def\de{\delta}

\def\lab{\label}

\providecommand{\dist}[1]{\mathrm{dist}[#1]}

\def\Sfrak{\mathfrak{S}}

\begin{document}

	\title[Generically sharp decay for a weak null wave system]{Generically sharp decay and blowing up at infinity for a weak null wave system}

 \author[S.~Dong]{Shijie Dong}
	\address{Southern University of Science and Technology, Shenzhen International Center for Mathematics, and Department of Mathematics, 518055 Shenzhen,  China.}
	\email{shijiedong1991@hotmail.com, dongsj@sustech.edu.cn}
	
	\author[S.~Ma]{Siyuan Ma}
	\address{State Key Laboratory of Mathematical Sciences, Academy of Mathematics and Systems Science, Chinese Academy of Sciences, Beijing 100190, China}
	\email{siyuan.ma@amss.ac.cn}
	
	\author[Y.~Ma]{Yue Ma}
	\address{Xi'an Jiaotong University, School of Mathematics and Statistics, 28 West Xianning Road, Xi'an Shaanxi 710049,  China.}
	\email{yuemath@xjtu.edu.cn}

	\author[X.~Yuan]{Xu Yuan}
	\address{State Key Laboratory of Mathematical Sciences, Academy of Mathematics and Systems Science, Chinese Academy of Sciences, Beijing 100190, China}
	\email{xu.yuan@amss.ac.cn}
	\begin{abstract}
We study a system of semilinear wave equations satisfying the weak null condition, which can be regarded as a simplified model for the Einstein vacuum equations. The main objective is to establish precise pointwise decay estimates, as both lower and upper bounds of decay, for small data solutions. Specifically, we show that the difference between the solution and its leading-order term is dominated by lower-order terms that decay faster in the retarded time variable $u = t-r$. Moreover, we prove that these pointwise decay estimates are sharp for a generic class of small initial data decaying sufficiently fast.

As applications of these estimates, we demonstrate that the energy of one component of the solution admits a lower bound that generically grows to infinity as $t \to +\infty$, which can be interpreted as “blowing up at infinity.” Furthermore, we verify that this component generically exhibits an energy cascade from high to low frequencies.
 \end{abstract}
	
    \maketitle
    \tableofcontents

 \section{Introduction}
 \subsection{Main results}
 In this article, we consider the following semilinear wave system satisfying the weak null condition,
 \begin{equation}\label{equ:main}
    \left\{
    \begin{aligned}
        -\Box \phi&=\left(\pt \psi\right)^{2},\quad \ (t, x)\in \{ (t,x):t^2 -r^2 \geq 1 \},
        \\
        -\Box \psi&=Q_{0}(\phi,\phi),\ \  (t, x)\in \{ (t,x):t^2 -r^2 \geq 1 \},
    \end{aligned}\right.
 \end{equation}
with initial data posed on the hypersurface $\mathcal{H}_1 := \{(t,x): t^2-r^2 =1 \}$,
 \begin{equation}\label{eq:ID}
 \begin{aligned}
 (\phi, \partial_t\phi)_{|\mathcal{H}_1} 
 =&(\phi(\sqrt{1+r^2}, x), \partial_t\phi(\sqrt{1+r^2}, x))
 = (\phi_0(x), \phi_1(x)),
 \\
  (\psi, \partial_t\psi)_{|\mathcal{H}_1} 
   =&(\psi(\sqrt{1+r^2}, x), \partial_t\psi(\sqrt{1+r^2}, x))
   = (\psi_0(x), \psi_1(x)).
 \end{aligned}
 \end{equation}
 Here, we denote $\Box$ the d’Alembertian operator 
 and $Q_{0}$ a null form:
 \begin{equation}
\lab{def:waveoperatorandQ0}
\Box=\partial_{\alpha}\partial^{\alpha}\quad \mbox{and}\quad Q_{0}(\phi,\phi)=\partial_{\alpha}\phi\partial^{\alpha}\phi,
 \end{equation}
 {where $\partial^{\alpha}=\eta^{\alpha\beta}\partial_{\beta}$ with $\eta^{\alpha\beta}$ being the inverse metric of the standard Minkowski metric $\eta_{\alpha\beta}$ with signature $(-,+,+,+)$.}
 This system is typically regarded as a simplified model for the vacuum Einstein equations in a wave gauge; see already Section \ref{subsect:VEEinwavegauge} for our motivation of studying the precise decay for the wave system \eqref{equ:main} to which the vacuum Einstein equations in a wave gauge can be reduced.

Denote $\| \cdot\|$ the usual $L^2$ norm in $\mathbb{R}^3$
of a given function. For a sufficiently regular function $f$, we denote $f_{\ell}$ with $\ell\in \mathbb{N}$, the $\ell$-th spherical harmonic mode of $f$. In particular, $f_{\ell=0}$ is the spherical mean of function $f$. See more details in Section \ref{SS:nota}.

Denote $r:=|x|$, 
and define the retarded and forward null coordinates
\begin{equation}
\lab{eq:coorduandv}
    u:=t-r\quad \mbox{and}\quad  v:=t+r.
\end{equation} 
Denote the future Cauchy development of the initial surface $\mathcal{H}_1$ by
\begin{equation}
\label{def:fullspacetimeregion}
\MM:=\{(t,x): u> 0 \ \ \mbox{and}\ \  v\geq u^{-1}\},
\end{equation} 
 and 
define a proper subset of the future null infinity by 
\begin{equation}
\lab{def:nullinfinitygeq0}
    \mathcal{I} := \{(t,x): (u,v)\in [0,+\infty)\times \left\{+\infty\right\}\}.
\end{equation} 
Also, we introduce the following definition of two space-time regions:
 \begin{definition}
 \label{def:regionsIandII:intro}
     Let $0<\delta\ll 1$ {be given}. We define the following two space-time regions:
     \begin{equation}
   \mathrm{I}:=\MM\cap\left\{r\le \frac{1}{2}u^{1-\delta}\right\}\quad 
         \mbox{and}\quad 
         \mathrm{II}:=\MM\cap\left\{r\ge \frac{1}{2}\exp({u^{\delta}})\right\}.
     \end{equation}
 \end{definition}

To state our main theorems on the asymptotic behavior of the solution in the full space-time region $\MM$, we define the following two explicit functions $\phi_{L}=\phi_{L}(u,v)$ and $\psi_{L}=\psi_{L}(u,v)$\footnote{Notice that in the following, $r=\frac{1}{2}(v-u)$ is a function of $(u,v)$.}: 
\begin{equation}
 \lab{eq:phibarandpsibar}
     \begin{aligned}
             \phi_{L}:=r^{-1}\left(\ln v-\ln u\right)\ \ \mbox{and}\ \  \psi_{L}:=r^{-1}\left(\frac{\ln u}{ u}-\frac{\ln v}{v}\right)
    \end{aligned}
\end{equation}
and introduce the new unknown 
   \begin{equation}
   \label{eq:definitionoftildepsi}
      \bar{\psi}:=\psi+\frac{1}{2}\phi^{2},
   \end{equation} 
  which,  in view of the system of equations \eqref{equ:main}, together with $\phi$ satisfies
\begin{equation}\label{equ:main-new}
    \left\{
    \begin{aligned}
        -\Box \phi&=\left(\pt \psi\right)^{2},
        \\
        -\Box \bar{\psi} &= \phi (\partial_t \psi)^2.
        \\
    \end{aligned}\right.
 \end{equation}

 \begin{theorem}\label{thm:main1}
 Let $N\in \mathbb{N}^{+}$ with $N\ge 6$. There exists {an} $\epsilon_{0}>0$ such that for all $\epsilon\in (0,\epsilon_{0})$ and all  initial data $(\phi_{0},\phi_1, \psi_{0}, \psi_1)$ satisfying the {following} smallness condition
\begin{equation}\label{est:smallness1}
\begin{aligned}
    \|\langle x\rangle^{{N}+2} \partial_x^{\leq N+1} (\phi_0,  \psi_0)\| 
    +  \|\langle x\rangle^{{N}+2} \partial_x^{\leq N} (\phi_1, \psi_1)\|   
    <\epsilon,
    \end{aligned}
\end{equation}
 the Cauchy problem~\eqref{equ:main}--\eqref{eq:ID} 
 admits a global-in-time solution $(\phi,\psi)$. Moreover, the following estimates of asymptotic behavior for $(\phi,\psi)$  are valid for $u\ge 2$.

 \begin{description}

         \item [ (i) Asymptotic behavior in \textrm{Region} $\mathrm{I}$]
         In the space-time region $\mathrm{I}$, we have 
          \begin{equation}
     \begin{aligned}
         \left|\phi(t,x)-\mathfrak{c}_{1}\phi_{L}\right|&\lesssim
        \ep \phi_{L}(\ln u)^{-1},\\
         \left|\psi(t,x)-\mathfrak{c}_{2}\psi_{L}\right|&\lesssim \ep \psi_{L}(\ln u)^{-1},
         \end{aligned}
         \end{equation}
{with} the constants {$\mathfrak{c}_{1}=\mathfrak{c}_{1}(\psi)\geq 0$ and $\mathfrak{c}_{2}=\mathfrak{c}_{2}(\phi,\psi)\in \R$}  given by 
         \begin{equation}
         \label{def:cfrak1and2:intro}
             \begin{aligned}
\mathfrak{c}_{1}&=\frac{1}{2}\int_{\mathcal{I}}
                 \left(\left(\pt (r\psi)\right)^{2}\right)_{\ell=0}(u,+\infty)\d u,\\
\mathfrak{c}_{2}&=\int_{\mathcal{I}}
                 \left(\frac{r\phi}{\ln v}
                 \left(\pt (r\psi)\right)^{2}\right)_{\ell=0}(u,+\infty)\d u.
             \end{aligned}
         \end{equation}

         \item [ (ii) Asymptotic behavior in \textrm{Region} $\mathrm{II}$]
         In the space-time region $\mathrm{II}$, we have
         \begin{equation}
     \begin{aligned}
         \left|\phi(t,x)-\mathfrak{c}_{3}\phi_{L}\right|&\lesssim
        \ep \phi_{L}(\ln u)^{-1},\\
        \left|\bar{\psi}(t,x)-r^{-1}{(r\psi)}(u,+\infty,\omega)+\mathfrak{c}_{4}r^{-1}v^{-1}\ln v\right|&\lesssim \ep r^{-1}v^{-1}u^{-\frac{\delta}{4}}\ln v.
         \end{aligned}
         \end{equation}
with the functions on sphere $\mathfrak{c}_{3}={\mathfrak{c}_{3}(\psi, \omega)}:\mathbb{S}^{2}\mapsto [0,+\infty)$ and $\mathfrak{c}_{4}={\mathfrak{c}_{4}(\phi,\psi, \omega)}:\mathbb{S}\mapsto \R$ given respectively by 
         \begin{equation}
         \label{def:cfrak3and4:intro}
             \begin{aligned}
                 \mathfrak{c}_{3}&=\frac{1}{2}\int_{\mathcal{I}}
                 \left(\pt (r\psi)\right)^{2}(u,+\infty{,\omega})\d u,\\
                 \mathfrak{c}_{4}&=\int_{\mathcal{I}}
                 \left(\frac{r\phi}{\ln v}
                 \left(\pt (r\psi)\right)^{2}\right)(u,+\infty{,\omega})\d u.
             \end{aligned}
         \end{equation}

           \item [ (iii) Asymptotic behavior of zeroth-order mode] In the space-time region $\MM$, we have for the zeroth-order mode $(\phi_{\ell=0},\bar{\psi}_{\ell=0})$, i.e., the radially symmetric part of $(\phi,\bar{\psi})$,  
     \begin{equation}
     \begin{aligned}
         \left|\phi_{\ell=0}(t,x)-\mathfrak{c}_{1}\phi_{L}\right|&\lesssim\ep 
         \phi_{L}(\ln u)^{-1},\\
         \left|\bar{\psi}_{\ell=0}(t,x)-\mathfrak{c}_{2}\psi_{L}\right|&\lesssim \ep \psi_{L}(\ln u)^{-1},
         \end{aligned}
         \end{equation}
while, for the zeroth-order mode $(\phi_{\ell=0},{\psi}_{\ell=0})$, we have, for $r\geq \frac{1}{2}u^{1+\delta}$ and $r\leq \frac{1}{2}u^{1-\delta}$ with an arbitrary small $\de>0$,
  \begin{equation}
     \begin{aligned}
         \left|\phi_{\ell=0}(t,x)-\mathfrak{c}_{1}\phi_{L}\right|&\lesssim\ep 
         \phi_{L}(\ln u)^{-1},\\
         \left|\psi_{\ell=0}(t,x)-\mathfrak{c}_{2}\psi_{L}\right|&\lesssim \ep \psi_{L}(\ln u)^{-1}.
         \end{aligned}
         \end{equation}
 \end{description}

In addition, there exists an open and dense\footnote{What we mean by ``open and dense" is with respect to a natural topology induced by the norm on the left-hand side of \eqref{est:smallness1}. See Section \ref{sec:generic}.} subset of initial data satisfying \eqref{est:smallness1} such that
\begin{itemize}
    \item $\cfrak_i\neq 0$ for $i=1,2$,
    \item and there exists an $\omega\in\mathbb{S}^2$ such that $\cfrak_i(\omega)\neq 0$ for $i=3,4$.
\end{itemize} 
That is, the above decay estimates in {$\mathrm{(i)}$-$\mathrm{(iii)}$} are generically sharp.
 \end{theorem}

In Theorem \ref{thm:main1}, the above {$\mathrm{(i)}$-$\mathrm{(iii)}$} are proved in Section \ref{sec:precise-decay}, and the last statement on the genericity of the decay estimates is established in Section \ref{sec:generic}.

Note that in Theorem \ref{thm:main1}, the precise decay of the zeroth-order mode $(\phi_{\ell=0}, \bar{\psi}_{\ell=0})$ in the full spacetime \textrm{region} $\mathrm{i}$s illustrated in {$\mathrm{(iii)}$}, while, for the full solution $(\phi, \psi)$,
as illustrated in $\mathrm{(i)}$ and $\mathrm{(ii)}$, we are only able to provide the precise pointwise decay in the regions $\mathrm{I}$ and $\mathrm{II}$. The main reason, that we can not derive its precise pointwise decay in the complementary region of $\mathrm{I}\cup\mathrm{II}$, is that all orders of the modes of the solution in fact decay at the same rate near the light cone, say $|r-t|\leq \frac{1}{2} t$, which does not allow us to sum up over all the modes to deduce the precise leading order behavior for the solution. 
Our main result for the higher order modes, say $\phi_{\ell}$ with $\ell\geq 1$, is as follows.

\begin{theorem}\label{thm:higher}
Consider the Cauchy problem \eqref{equ:main}--\eqref{eq:ID}, and let the smallness condition \eqref{est:smallness1} for the initial data hold. For positive integers $\ell\leq N-3$, define $\mathcal{M}_{\ell} :=\MM\cap\{ r\geq u^{1-\delta_\ell} \}$ with $\delta_\ell \leq \min\{ \frac{\delta}{2},\frac{1}{4\ell + 4}\}$. Then in $\mathcal{M}_{\ell}$ and for $u\geq 2$ large, we have
\begin{align}
    \bigg|\phi_\ell(t, x) - \frac{(-1)^{\ell+1}}{2r} \mathcal{C}_\ell \mathcal{D}_\ell(ur^{-1}) \bigg|
    \lesssim_\ell
    \ep r^{-1} u^{-\delta_{\ell}},
\end{align}
in which
\begin{align*}
    \mathcal{C}_\ell
    =
    \int_{\mathcal{I}} \left((\partial_t \Psi)^2\right)_\ell,
    \qquad
    \mathcal{D}_\ell(ur^{-1})
    =\int_{ur^{-1}}^{+\infty} {(z-ur^{-1})^{\ell} \over z^{\ell+1} (1+ {z\over 2})^{\ell+1}} \, dz.
\end{align*}
\end{theorem}
 
 The proof of Theorem \ref{thm:higher} is contained in Section \ref{sec:appendix}.
 
\begin{remark}
    One can analyze the behavior of $\mathcal{D}_\ell({u r^{-1}})$, which reads
    \begin{align*}
        |\mathcal{D}_\ell({u r^{-1}})| 
        \simeq & u^{-1-\ell \kappa_1},
        \qquad
        &&\text{ for }  r\simeq u^{1-\kappa_1}, \, 0<\kappa_1\leq {\delta\over 4},
        \\
        |\mathcal{D}_\ell({u r^{-1}})| 
        \simeq &1,
        \qquad
        &&\text{ for }  r \simeq u,
        \\
        |\mathcal{D}_\ell({u r^{-1}})| 
        \simeq & \kappa_2 \ln u,
        \qquad
        &&\text{ for }  r \geq u^{1+\kappa_2}, \, \kappa_2>0.
    \end{align*}
    It demonstrates how the decay grows as $r$ goes from near origin to $r\simeq u$ and to $r \gg u$. 
    
    The constraint $\ell \leq N-3$ is due to the analysis in the proof where we use pointwise information of the solution $(\phi, \psi)$ and thus lose derivatives.
\end{remark}

\subsection{Application to generic energy cascade and blow-up at infinity}

We are now ready to state the following theorem on the generic blow-up at infinity and energy cascade of the field $\phi$, which is proved in Section \ref{sec:growth}.

\begin{theorem}[Generic energy cascade and blow-up at infinity]\label{thm:blow-up} 
Let the smallness condition \eqref{est:smallness1} for the initial data hold, and assume moreover that the initial data are compactly supported\footnote{By saying that the initial data is compactly supported on $\mathcal{H}_1$, we mean there exists an $R_0>0$ such that the initial data can be extended in a way such that it is defined on $(\mathcal{H}_1\bigcap \{ r\leq R_0+1 \})\bigcup \{t=\sqrt{(R_0 + 1)^2 + 1}, \, r\geq R_0 +1 \}$ and vanishes for $r\geq R_0$.}. Then, the following statements are valid.
 \begin{description}
\item [ (i) Generic energy cascade and blow-up at infinity of $\phi$]
         The following lower and upper bounds 
         \begin{equation}
         \begin{aligned}
              \mathfrak{c}_1 t^{1\over 2}&\lesssim \| \phi \| \lesssim \ep t^{\frac{1}{2}} \ln t,\\
              \mathfrak{c}_{5}\ln t&\lesssim \| \partial \phi \| \lesssim \epsilon \ln t,
              \end{aligned}
         \end{equation}
         with $\mathfrak{c}_5$ given by
         \begin{equation}
         \label{def:cfrak5constant}
    \mathfrak{c}_{5}=\bigg( \int_{\mathbb{S}^2} \int_0^{+\infty} (\partial_t(r\psi))^4(u,+\infty,\omega)\d u \d \omega \bigg)^{1\over 2},
\end{equation}
hold for sufficiently large $t$ and imply that the following holds generically\footnote{In view of Theorem \ref{thm:main1}, $\mathfrak{c}_1 > 0$ holds generically, that is, $\mathfrak{c}_1 > 0$ holds for an open and dense subset of the admissible initial data set. This implies that $\mathfrak{c}_5>0$ holds generically as well.} 
\begin{itemize}
    \item the component $\phi$ exhibits energy cascade from high to low frequencies,
    \item and both $L^2$-norms of $\phi$ and $\partial\phi$  blow up at time infinity.
\end{itemize} 

       \smallskip
 \item [ (ii) Uniform bounds of $\psi$]
As a comparison, the field $(\psi,\partial \psi)$ admits uniform-in-time $L^2$ norms
\begin{equation}
    \| \psi \| + \| \partial\psi \| \lesssim \epsilon.
\end{equation}
\end{description}
\end{theorem}

\begin{remark}
As we know, most of the results concerning the estimates of the norms of solutions prove only upper bounds. Here, our result provides a lower bound that grows to infinity generically, which implies the generic blow-up at infinity phenomenon for both $\|\phi\|$ and $\|\partial\phi\|$.
\end{remark}

\begin{remark}
   The statement that the lower bound of $\|\partial \phi\|$ generically grows to infinity as time evolves to infinity also asserts that the component $\phi$ does not scatter in the energy space generically.
\end{remark}

\begin{remark}
By a similar argument, one can remove the compact support requirement for the initial data and show the natural energy of $\phi$ on hyperboloids (see Section \ref{SS:nota} for its definition) also admits a lower bound that grows to infinity as the hyperbolic time $s=\sqrt{t^2-r^2}$ goes to $+\infty$.
\end{remark}

\subsection{Relation to the Einstein equations in the wave gauge}
\label{subsect:VEEinwavegauge}

It is well-known (c.f. \cite[Chapter VI.7]{YCB}) that the vacuum Einstein equations reduce to a system of quasilinear wave equations enjoying a weak null structure within a specially constructed coordinate chart, called the wave gauge. Specifically, imposing a wave gauge in which the coordinates satisfy the wave equations
\begin{equation}\label{eq1-19-jan-2026-YM}
\Box_{g} x^{\alpha} = 0,
\end{equation}
the vacuum Einstein equations are transformed into a system of wave equations
\begin{equation}\label{eq2-19-jan-2026-YM}
g^{\mu\nu}\partial_{\mu}\partial_{\nu}g_{\alpha\beta} = F_{\alpha\beta}(g,g,\partial g,\partial g) 
\end{equation}
with $F_{\alpha\beta}$ a collection of quadratic and higher-order (i.e., at least cubic) terms of $g,\partial g$. Subtracting the Minkowski metric in the sense of (detail can be found in, for instance,  \cite[page 15]{LindbladRodnianski2010})
$$
h_{\alpha\beta} := g_{\alpha\beta} - \eta_{\alpha\beta},
$$ 
the wave system \eqref{eq2-19-jan-2026-YM} can be reduced to 
\begin{equation}\label{eq3-19-jan-2026-YM}
\Box h_{\alpha\beta} = \frac{1}{4}\partial_{\alpha}\big(\eta^{\mu\nu}h_{\mu\nu}\big)\partial_{\beta}\big(\eta^{\mu\nu}h_{\mu\nu}\big) + S_{\alpha\beta},
\end{equation}
where $S_{\alpha\beta}$ are composed of semilinear quadratic terms of $\partial h_{\alpha\beta}$, quasilinear terms of type $h\partial\partial h$ and higher-order terms\footnote{Those semilinear terms violating the null conditions and the quasilinear terms can be treated via other techniques, including the wave gauge condition, integration by parts, etc.}.

Furthermore, by contracting the wave equation \eqref{eq3-19-jan-2026-YM} with the Minkowski inverse metric $\eta^{\alpha\beta}$, we obtain
\begin{equation}\label{eq4-19-jan-2026-YM}
\Box(\eta^{\alpha\beta}h_{\alpha\beta}) = \frac{1}{4}\eta^{\alpha\beta}\partial_{\alpha}\big(\eta^{\mu\nu}h_{\mu\nu}\big)\partial_{\beta}\big(\eta^{\mu\nu}h_{\mu\nu}\big) + \eta^{\alpha\beta}S_{\alpha\beta}
\end{equation}
with the right-hand side containing only null semilinear quadratic terms, quasilinear terms, and higher-order terms.
In view of  \eqref{eq3-19-jan-2026-YM} and \eqref{eq4-19-jan-2026-YM}, the components $h_{00}$ and $\eta^{\alpha\beta}h_{\alpha\beta}$  solve the following coupled wave system
\begin{equation}
\label{eq:h00andtrh:semiwaveweaknull:vee}
\left\{
\aligned
&\Box h_{00} = \frac{1}{4}\partial_t\big(\eta^{\mu\nu}h_{\mu\nu}\big)\partial_t\big(\eta^{\mu\nu}h_{\mu\nu}\big) + S_{00},
\\
&\Box(\eta^{\alpha\beta}h_{\alpha\beta}) = \frac{1}{4}\eta^{\alpha\beta}\partial_{\alpha}\big(\eta^{\mu\nu}h_{\mu\nu}\big)\partial_{\beta}\big(\eta^{\mu\nu}h_{\mu\nu}\big) + \eta^{\alpha\beta}S_{\alpha\beta}
\endaligned
\right.
\end{equation}
which enjoys a similar structure as \eqref{equ:main} once making the replacement $(h_{00}, \eta^{\alpha\beta}h_{\alpha\beta})\to (\phi,\psi)$, with the null quadratic term $Q_0(\phi,\phi)$ in the equation of $\psi$ standing for the null terms on $h_{00}$ contained in 
$\eta^{\alpha\beta}S_{\alpha\beta}$. 
The decay rates of the rest components of $h_{\alpha\beta}$ are not slower than that of $h_{00}$, i.e., the dominant behavior of the metric perturbation will be driven by the behavior of $h_{00}$, an analogue of the scalar $\phi$ in the system of wave equations \eqref{equ:main} satisfying the weak null condition. This motivates us to consider the wave system \eqref{equ:main} satisfying the weak null condition.

\subsection{Relevant literature}
The study of nonlinear wave equations has a long history, and we briefly review the existing literature relevant to this paper.  According to the pioneering work of John \cite{John}, it is known that general quadratic nonlinearities might lead to small data blow-up in finite time, while nonlinearities under the null condition by Klainerman \cite{Klainerman86} and Christodoulou \cite{Christodoulou86} ensure small data global existence. Later on, the null condition on the quadratic nonlinearities to ensure small data global existence is relaxed; see the weak null condition of Lindblad-Rodnianski \cite{LindbladRodnianski2005,LindbladRodnianski2010}, the non-resonant condition of Pusateri-Shatah \cite{PusateriShatah2013}, etc. In particular, we mention some further results \cite{Alinhac2003, Lindblad2008, Katayama2015, DengPusateri2020, Keir18} related to the weak null condition.

Due to dispersion, the magnitude of linear homogeneous waves in $\mathbb{R}^{1+3}$ decays as time evolves. In addition to global existence, the sharp pointwise decay of the solution is important in the study of wave equations. In this direction, for linear waves on curved background, we recall the work of Angelopoulos-Aretakis-Gajic \cite{angelopoulos2021late} and Hintz \cite{Hintz} on scalar field, the work of Ma-Zhang \cite{MZ21PLKerr} and Millet \cite{Millet23} on Teukolsky equations in Kerr spacetimes, and the work of Luk-Oh \cite{LuOh} on the precise decay for waves on dynamical backgrounds in odd spatial dimensions.  For nonlinear waves, we recall the work of Luk-Oh \cite{LuOh19} on the Einstein-Maxwell-scalar field system under spherical symmetry, the work of Deng-Pusateri \cite{DengPusateri2020}, Yu \cite{Yu21, Yu24, Yu}, and Luk-Oh-Yu \cite{LuOhYu} on the quasilinear wave equation $-\Box \phi = \phi \Delta \phi$ satisfying the weak null condition, the work of Looi-Xiong \cite{LX25} on the asymptotic decay and expansions for semilinear wave equations, and the work of Looi-Tohaneanu \cite{Looi2022Nullcondition} and Dong-Ma-Ma-Yuan \cite{DMMY} on a quasilinear wave equation with the null condition. We remark that for nonlinear wave equations, there are very few results regarding the generic precise decay of the solution, except the results \cite{LuOh19, DMMY, LuOhYu}.

\subsection{Major difficulties and key ideas} 
 
In our previous paper \cite{DMMY}, we treated a quasilinear wave equation satisfying the classical null condition and derived its precise pointwise decay. Compared with this and other existing results, several new difficulties and phenomena arise in our current paper. 

First, in contrast to the other results on the sharp decay for wave equations, the radiation field of the component $\phi$ is not well-defined at the future null infinity $\mathcal{I}$. The (Friedlander) radiation field of $\phi$ is given by $\Phi = r\phi$, and its value on $\mathcal{I}$ is typically (see, for instance, in \cite{LuOh, DMMY}) used to describe the precise pointwise behavior of the field. Due to the weak decay of the nonlinearity in the wave equation of $\phi$, the field $\Phi=r\phi$ in fact blows up logarithmically in $r$ towards the null infinity\footnote{Such a phenomenon, as demonstrated in \cite{DengPusateri2020,Yu24,Yu}, occurs in the quasilinear wave equation $\Box\phi=\phi\Delta \phi$ as well. It is also related to the fact \cite{YCB73,LindbladRodnianski2010} that the metric $g$ solving the vacuum Einstein equation under a wave gauge diverges logarithmically from the Minkowski cones.}. This imposes one of the major difficulties in our analysis for the field $\phi$. 

Second, the higher-order modes of the solutions $(\phi, \psi)$ do not decay faster than the zeroth-order mode $(\phi_{\ell = 0}, \psi_{\ell=0})$. The works \cite{LuOh19,angelopoulos2021late,Hintz,MZ21PLKerr,DMMY,Millet23,LuOh} on sharp decay for waves rely on the fact that the higher-order modes $\phi_{\ell\geq 1}$ decay strictly faster than the  $0$-th mode, indicating that the leading order term of the $0$-th mode $\phi_{\ell=0}$ is the leading order term of the solution $\phi$ itself, and further reducing the proof of the sharp decay for $\phi$ to a much simpler proof of the sharp decay for $\phi_{\ell=0}$. However, this property fails for \eqref{equ:main}. Indeed, we demonstrate that in a large spacetime region (in particular, in a spacetime region where $u \lesssim v$), the higher modes $\phi_{\ell}$ (resp. $\psi_\ell$) with $\ell\geq 1$ decay at the same rate as $\phi_{\ell=0}$ (resp. $\psi_{\ell=0}$);
This complicated phenomenon makes the analysis delicate.

In establishing the precise pointwise decay estimates for $(\phi,\psi)$, one key ingredient is the application of the nonlinear transformation $\bar{\psi}=\psi+\frac{1}{2}\phi^{2}$; see \eqref{eq:definitionoftildepsi} and \eqref{equ:main-new}. Upon this transformation, the nonlinearity on the right-hand side of the wave equation \eqref{equ:main-new} has faster decay than the nonlinearity on the right-hand side of the wave equation \eqref{equ:main} of $\psi$, which enables us to more easily derive the precise decay estimates.

Another key ingredient is the detailed properties of the term 
\begin{equation*}
X(\Phi,\Psi)=(\partial_t -\partial_r)\Phi - \ln v (\partial_t \Psi)^2,
\end{equation*}
 with $\Psi=r\psi$ being the radiation field of $\psi$, which may also be viewed as a nonlinear transformation. We observe that there is a cancellation, up to terms decaying faster, within this term, which eliminates a $\ln v$ growth in its pointwise bounds. 
 In addition, the equation of $(\partial_t +\partial_r)X(\Phi,\Psi) = \cdots$ (see \eqref{eq:non-trans} for the details) allows us to infer refined estimates for this term, which in turn can be used to infer refined estimates for $\Phi/\ln v$ and hence compute the leading order term in the late-time asymptotics of $\bar{\psi}$. 
 
We also comment that these refined estimates for the term $X(\Phi,\Psi)$ play a vital role in showing that the natural energy for $\phi$ diverges logarithmically as time goes to infinity.

It is worth mentioning that the proof for the genericity part also differs much from the existing literature. When one extracts the leading part of, say $\phi$, it is important to verify that the constant in front, say $\mathfrak{c}_1$ as given in Theorem \ref{thm:main1}, does not vanish so as to guarantee that it is indeed the leading order one. In the present paper, we set the initial data to decay suitably fast such that there is no contribution from the initial data to the coefficients $\{\mathfrak{c}_i\}_{i=1,2,3,4}$. Indeed, the constants $\{\mathfrak{c}_i\}_{i=1,2}$ and the functions on sphere $\{\mathfrak{c}_i\}_{i=3,4}$ are purely determined by the integral of the nonlinear terms on the future null infinity $\mathcal{I}$. Therefore, it is a nonlinear problem to verify that the constants or functions $\{\mathfrak{c}_i\}_{i=1,2,3,4}$ are generically nonzero.

 To establish the genericity argument, we rely on a hidden relation among the constants or functions on spheres $\{\mathfrak{c}_i\}_{i=1,2,3,4}$, reducing the full problem of proving generic nonvanishing property for all the constants or functions $\{\mathfrak{c}_i\}_{i=1,2,3,4}$ to showing only the generic non-vanishing of the constant $\mathfrak{c}_1$. Assuming otherwise that there exists a pair of initial data of $(\phi,\psi)$ such that $\cfrak_1=0$ holds in a deleted $\epsilon_1$-neighborhood of such an initial data. Denote $(\check{\phi},\check{\psi})$  the difference of the perturbed solution to a general initial data in such a deleted $\epsilon_1$-neighborhood from the solution $(\phi,\psi)$.  Exploiting the fact that $\mathfrak{c}_1(\psi)=0$ implies a faster decay in $v$ for the component $\partial \psi$, we can perform an energy estimate for $\check{\phi}$ to deduce a bound for the standard energy of $\check{\phi}$
\begin{equation*}
    \sup_{s\in [1, +\infty)} \mathcal{E}(s, \check{\phi})
    \lesssim
    \mathcal{E}(1, \check{\phi})
    +
    \epsilon \sup_{s\in [1, +\infty)} \mathcal{E}(s, \check{\psi}).
\end{equation*}
Also, we can achieve an energy bound for $\check{\psi}$
 \begin{equation*}
    \sup_{s\in [1, +\infty)} \mathcal{E}(s, \check{\psi})
    \lesssim
    \epsilon \sup_{s\in [1, +\infty)} \mathcal{E}(s, \check{\phi}).
\end{equation*}
Combining the above two estimates, we obtain
\begin{align*}
  \mathcal{E}(1, \check{\psi})
    \lesssim
    \epsilon \mathcal{E}(1, \check{\phi}),
\end{align*}
which manifestly can not hold everywhere in a deleted $\epsilon_1$-neighborhood of the initial data of $(\phi,\psi)$ and hence completes the proof of genericity.

\subsection{Organization}

In Section \ref{sect:PreliminaryandAlmostsharp}, we introduce the notation and prove global existence of the wave system \eqref{equ:main} with almost sharp decay estimates for the solution.
Then, we derive in Section \ref{sec:precise-decay} the precise decay of the solution and 
establish in Section \ref{sec:generic} that the leading order terms in Theorem \ref{thm:main1} are generically non-vanishing, which hence proves Theorem \ref{thm:main1}. 
Next, Theorem \ref{thm:higher} on the sharp decay for the higher-order modes of $\phi$ is proved in Section \ref{sec:appendix}. Finally, in Section \ref{sec:growth}, we verify the growing lower bound of the energy of the field $\phi$ and conclude the proof of Theorem \ref{thm:blow-up}.

 \subsection*{Acknowledgement}

The author Shijie Dong would like to acknowledge the support from the National Natural Science Foundation of China (Grant Nos. 12401280 and 12431007) and Guangdong Basic and Applied Basic Research Foundation (Grant Nos. 2025A1515012652 and 2023A1515110944).  The author Yue Ma is supported by the Fundamental Research Funds for the Central Universities (Xi'an Jiaotong University, Grant Nos. xzy012023034 and xzy022025046).

\section{Preliminaries and almost sharp decay estimates}
\label{sect:PreliminaryandAlmostsharp}

 \subsection{Notation and conventions}\label{SS:nota}
We work in  {a $(1+3)$-}dimensional Minkowski spacetime $\R^{1+3}$. {Denote} one point {in $\R^{1+3}$} by {$(t,x)$, with $t=x_0$ and $x=(x_{1},x_{2},x_{3})$,} and {denote} its spatial radius by $r{=|x|=}\sqrt{x_{1}^{2}+x_{2}^{2}+x_{3}^{2}}$. Hence, the spacetime equals $\R^{1+3}=\left\{(t,r,\omega):t\in \R,\ r\in [0,+\infty)\ \mbox{and}\ \omega\in \mathbb{S}^{2}\right\}$.
Recall that we denote
\begin{equation*} 
-\Box=\partial_{t}^{2}-\Delta\quad \mbox{with}\ \ \Delta=\partial_{r}^{2}+\frac{2}{r}\partial_{r}+\frac{1}{r^{2}}\Deltas.
\end{equation*}
Greek indices $\left\{\mu,\nu,\dots\right\}$ range over $\left\{0,1,2,3\right\}$, Roman indices $\left\{a,b,\dots\right\}$ range over $\left\{1,2,3\right\}$, and the Einstein summation convention for repeated upper and lower indices is applied throughout this paper.
 
\smallskip
Following Klainerman~\cite{Klainerman86}, we introduce the following vector fields:

\begin{enumerate}
	
		\smallskip
	\item [(i)]  Translations: $\partial_{\alpha}=\partial_{x_{\alpha}}$, for $\alpha=0,1,2,3$.
	
	\smallskip
	\item [(ii)] Scaling vector field: $S=t\partial_{t}+x^{a}\partial_{a}$.
	
	\smallskip
	\item [(iii)] Hyperbolic rotations: $L_{a}=x_{a}\partial_{t}+t\partial_{a}$, for $a=1,2,3$.
	
	\smallskip
	\item [(iv)] Spatial rotations: $\Omega_{ab}=x_{a}\partial_{b}-x_{b}\partial_{a}$, for $1\le a<b\le 3$.
\end{enumerate}
To simplify the notation, we denote 
\begin{equation*}
    \Gamma = (S,\partial_{0},\partial_{1},\partial_{2},\partial_{3}, L_1, L_2, L_3, \Omega_{12}, \Omega_{13}, \Omega_{23}).
\end{equation*}
Moreover, we denote 
\begin{equation*}
\partial=\left(\partial_{0},\partial_{1},\partial_{2},\partial_{3}\right),\ 
L=\left(L_{1},L_{2},L_{3}\right)\ \mbox{and}\
\Omega=\left(\Omega_{12},\Omega_{13},\Omega_{23}\right).
\end{equation*}
On the other hand, we use $I$ and $J$ to denote general multi-indices in   $\mathbb{N}^{11}$. 

\smallskip
Recall from \eqref{eq:coorduandv} the null coordinates $(u,v)=(t-r,t+r)$. 
Let $\partial_u$ and $\partial_v$ be the coordinate derivatives in the coordinate system $(u,v, \omega)$ and, thus, we have
\begin{equation*}
    \partial_{u}=\frac{1}{2}(\pt -\pr)\quad \mbox{and}\quad 
    \partial_{v}=\frac{1}{2}(\pt +\pr ).
\end{equation*}
For abbreviation, we also define
\begin{equation}
\lab{def:UandV}
    U=2\partial_{u}=\pt -\pr\quad \mbox{and}\quad 
    V=2\partial_{v}=\pt +\pr . 
\end{equation}
From \cite[Theorem 5.10]{AlinhacBook}, for any sufficiently regular function $f=f(t,x)$, 
\begin{equation}\label{est:t-rpoint}
   u \left|\partial f\right|+u\left|Uf\right|
   +
    v|Vf|\lesssim |\Gamma f|.
\end{equation}
Recall also that, for any sufficiently regular function $f=f(t,x)$, we have 
\begin{equation}\label{est:Omega}
    \left|r^{-1}\Omega f\right|\lesssim v^{-1}\left|\Gamma f\right|\quad \mbox{and}\quad 
    \left|r^{-2}\Deltas f\right|\lesssim v^{-2}\sum_{1\leq |I|\leq 2}\left|\Gamma^I f\right|.
\end{equation}

\smallskip
Define the hyperbolic time $s=\sqrt{t^{2}-r^{2}}$ with $s\ge 1$, and denote constant-$s$ hyperboloidal hypersurfaces by
\begin{equation*}
    \mathcal{H}_{s}{=}\left\{(t,x)\in \mathbb{R}^{1+3}:s^{2}=t^{2}-r^{2}\right\}, \quad \mbox{for any}\ s\ge 1.
\end{equation*}
For each hypersurface $\mathcal{H}_s$, we denote by $\Vec{n}$ the future directed  unit normal vector and $\mathrm{d}\sigma$ {the} volume element (with respect to the Euclidean metric in $\mathbb{R}^{1+3}$),  that is, 
\begin{equation}\label{equ:normal}
\vec{n}\mathrm{d}\sigma =\left(1,-\frac{x_{1}}{t},-\frac{x_{2}}{t},-\frac{x_{3}}{t}\right)\mathrm{d}x.
\end{equation}
Moreover, for any $1\le s_{1}<s_{2}<+\infty$, we denote
\begin{equation*}
    \mathcal{H}_{[s_{1},s_{2}]}{=}\bigcup_{s=s_{1}}^{s_{2}}\mathcal{H}_{s}=\left\{(t,x)\in \mathbb{R}^{1+3}:t^{2}-r^{2}\in \left[s_{1}^{2},s_{2}^{2}\right]\right\}.
\end{equation*}

For a sufficiently regular function $f=f(t,x)$ defined on $\mathcal{H}_{s}$, we define
\begin{equation*}
    \int_{\mathcal{H}_{s}}f(t,x)\d x=\int_{\R^{3}}f(\sqrt{s^{2}+r^{2}},x)\d x.
\end{equation*}
In addition, the Sobolev norms of the function $f=f(t,x)$ on $\mathcal{H}_s$ are defined by 
\begin{equation*}
    \|f\|_{L^{p}(\mathcal{H}_{s})}^{p}=\int_{\mathcal{H}_{s}}|f(t,x)|^{p}\d x \ \  {\mbox{for}\ p\in [1,+\infty)}\quad\mbox{and}\ \ 
    \|f\|_{L^{\infty}(\mathcal{H}_{s})}=\sup_{\mathcal{H}_{s}}|f(t,x)|.
\end{equation*}

 To establish the energy estimate for the 3D inhomogeneous linear wave equation for future reference, we introduce the following standard energy $\mathcal{E}$ on the hyperboloidal hypersurface $\mathcal{H}_{s}$,
 \begin{equation}
 \label{def:standardenergy}
     \mathcal{E}(s,f){:=}\int_{\mathcal{H}_{s}}\left((s/t)^{2}|\partial_{t}f|^{2}+(1/t)^{2}|L f|^{2}\right)\mathrm{d}x.
 \end{equation}
In view of the above definition of $\mathcal{E}$, we directly have 
 \begin{equation}\label{est:L2energy}
     \|(s/t)\partial f\|_{L^{2}(\mathcal{H}_{s})}+\|Vf\|_{L^{2}(\mathcal{H}_{s})}+
      \|(1/t)L f\|_{L^{2}(\mathcal{H}_{s})}\lesssim \mathcal{E}(s,f)^{\frac{1}{2}}.
 \end{equation}
 We also introduce the following conformal energy $\mathcal{E}_{\rm{con}}$ on $\mathcal{H}_{s}$,
 \begin{equation*}
     \mathcal{E}_{\rm{con}}(s,f):=\int_{\mathcal{H}_{s}}
     \left(\left(Kf+2f\right)^{2}
     +\sum_{a=1}^{3}(s{\underline{\partial}}_{a}f)^{2}
     \right)\d x.
 \end{equation*}
 Here, we denote the conformal vector field
 \begin{equation*}
     K=s\bar{\partial}_{s}+2x^{a} {\underline{\partial}}_{a}\quad \mbox{with}\quad 
      \bar{\partial}_{s}=\left(\frac{s}{t}\right)\partial_{t} \ \ \mbox{and} \ \
     {\underline{\partial}}_{a}=\frac{x_{a}}{t}\partial_{t}+\partial_{a}.
 \end{equation*}
 From the classical Hardy inequality, we also have
 \begin{equation}\label{est:L2conenergy}
     \left\|s(s/t)^{2}\partial f\right\|_{L^{2}(\mathcal{H}_{s})}+\|(s/t)f\|_{L^{2}(\mathcal{H}_{s})}\lesssim \mathcal{E}_{\rm{con}}(s,f)^{\frac{1}{2}}.
 \end{equation}
 
Let ${\mathrm{d}\omega}$ be the volume element on unit sphere $\mathbb{S}^{2}$. For any sufficiently regular function $f=f(t,r,\omega)$, we set 
\begin{equation*}
    f_{\ell=0}(t,r):=\frac{1}{4\pi}\int_{\mathbb{S}^{2}}f(t,r,\omega){ \mathrm{d}\omega}.
\end{equation*}
Indeed, the subscript $\ell=0$ indicates that the above definition is the $\ell=0$ spherical harmonic mode of $f$ when the function $f$ is decomposed into spherical harmonics on $\mathbb{S}^{2}$. 
In addition, we set 
\begin{equation*}
    f_{\ge 1}(t,r,\omega):=f(t,r,\omega)-f_{\ell=0}(t,r).
\end{equation*}
More generally, for any $(\ell,m)\in {\mathbb{N}\times\mathbb{Z}}$ with $|m|\le \ell$, we denote by $Y_{\ell }^{m}=Y_{\ell}^{m}(\omega)$ the normalized $(\ell,m)$-order spherical harmonics, that is, 
\begin{equation*}
    -\Deltas Y_{\ell}^{m}(\omega)=\ell(\ell+1)Y_{\ell}^{m}(\omega)\quad \mbox{and}\quad 
    \int_{\mathbb{S}^{2}}Y_{\ell}^{m}(\omega)Y_{\ell'}^{m'}(\omega){ \mathrm{d}\omega}=\delta_{\ell \ell'}\delta_{m m'}.
\end{equation*}
In addition, for any sufficiently regular function $f=f(t,r,\omega)$ and $(\ell,m)\in {\mathbb{N}\times\mathbb{Z}}$ with $|m|\le \ell$, we set 
\begin{equation*}
    {f_{\ell m}}(t,r)=\int_{\mathbb{S}^{2}}f(t,r,\omega)Y_{\ell}^{m}(\omega){ \mathrm{d}\omega}\quad \mbox{and}\quad 
    f_{\ell}(t,r,\omega)=\sum_{m=-\ell}^{\ell}f_{\ell m}(t,r)Y_{\ell}^{m}(\omega).
\end{equation*}

 \subsection{Elementary analytic estimates}
 In this subsection, we recall several elementary estimates on the vector fields, the null form and the 3D {inhomogeneous} linear wave equations. First, we recall the following Sobolev-type inequality from{~\cite[Proposition 1]{KlainWave} and~\cite[Lemma 7.6.1]{HomanderBook}}. 

 \begin{lemma}
     [Klainerman-Sobolev inequality]\label{le:Klain}
     Let $f=f(t,x)$ be a {sufficiently }regular function defined on $\mathcal{H}_{s}$ with $s\in  [1,+\infty)$. Then we have 
     \begin{equation*}
         \sup_{\mathcal{H}_{s}}\big|t^{\frac{3}{2}}f(t,x)\big|
         \lesssim \sum_{|K|\le 2}\left\|L^{K}f\right\|_{L^{2}(\mathcal{H}_{s})}.
     \end{equation*}
 \end{lemma}

Second, we introduce the following estimate {on} the null form $Q_{0}$ given as in \eqref{def:waveoperatorandQ0}, which follows directly from the following decomposition of $Q_{0}$: 
\begin{equation*}
 \begin{aligned}
  Q_{0}(f,g)=&-(s/t)^{2}(\pt f)(\pt g)+\sum_{a=1}^{3}(t^{-1}L_{a}f)(t^{-1}L_{a}g)\\
  &-(x^{a}/t)(\pt f)(t^{-1}L_{a}g)
  -(x^{a}/t)(\pt g)(t^{-1}L_{a}f).
  \end{aligned}
 \end{equation*}
 
 \begin{lemma}[Estimate for the null form {$Q_0$}]\label{le:null}
     Let $f=f(t,x)$ and $g=g(t,x)$ be two sufficiently regular functions defined on $\mathcal{H}_{s}$ with $s\in [1,+\infty)$. Then we have 
     \begin{equation*}
         \left|Q_{0}(f,g)\right|\lesssim 
         (s/t)^{2}
         |\partial_{t}f||\partial_{t}g|
         +t^{-1}
         \left(\left|Lf\right|\left|\partial g\right|
         +\left|\partial f\right|\left|L g\right|
         \right).
     \end{equation*}
 \end{lemma}

Next, we recall the energy-type estimates for the 3D linear wave equation. The proof is similar to the case of flat foliation of the spacetime (see \emph{e.g.}~\cite[Chapter 6]{AlinhacBook}), but we give a sketch of it for the sake of completeness and for the reader’s convenience.
 \begin{lemma}\label{le:energy}
 Let $f=f(t,x)$ be the solution to the Cauchy problem
 \begin{equation*}
     -\Box f=F\quad \mbox{with}\quad (f,\pt f)_{|\mathcal{H}_{1}}=(f_{0},f_{1}).
     \end{equation*}
 Then for any $s\geq 1$, we have 
 \begin{equation*}
 \begin{aligned}
     \mathcal{E}(s,f)^{\frac{1}{2}} &\lesssim
     \mathcal{E}(1,f)^{\frac{1}{2}}+\int_{1}^{s}\|F\|_{L^{2}(\mathcal{H}_{\tau})}\mathrm{d} \tau,\\
     \mathcal{E}_{\rm{con}}(s,f)^{\frac{1}{2}} &\lesssim
     \mathcal{E}_{\rm{con}}(1,f)^{\frac{1}{2}}+\int_{1}^{s}\tau\|F\|_{L^{2}(\mathcal{H}_{\tau})}\mathrm{d} \tau.
     \end{aligned}
 \end{equation*}
     \end{lemma}

     \begin{proof}
         First, we can rewrite the product $(-\Box f)\pt f$ as the divergence form,
         \begin{equation*}
             (-\Box f)\partial_{t}f=\frac{1}{2}\pt \left((\pt f)^{2}+|\nabla f|^{2}\right)-\partial^{a}\left(\partial_{t}f\partial_{a}f\right).
         \end{equation*}
         Integrating the above identity over $\mathcal{H}_{[1,s]}$, and then using~\eqref{equ:normal}, we have 
         \begin{equation*}
             \mathcal{E}(s,f)-\mathcal{E}(1,f)=-2\int_{1}^{s}\int_{\mathcal{H}_{\tau}}(\tau/t)(\pt f)F\mathrm{d}x\mathrm{d}\tau,
         \end{equation*}
         which, by differentiating in $s$ and applying the Cauchy-Schwarz inequality, leads us to
          \begin{equation*}
             \mathcal{E}(s,f)^{1\over 2} {\mathrm{d}\over \mathrm{d}s}\mathcal{E}(s,f)^{1\over 2}=-\int_{\mathcal{H}_{s}}(s/t)(\pt f)F\mathrm{d}x\Longrightarrow {
             {\mathrm{d}\over \mathrm{d}s}}\mathcal{E}(s,f)^{1\over 2}
             \leq \|F\|_{L^{2}{(\mathcal{H}_{s})}}.
         \end{equation*}  
         Integrating the above estimate over $[1,s]$, we complete the proof for the estimate of $\mathcal{E}$.

         \smallskip
         Second, the proof for the conformal energy estimate relies on the use of a non-spacelike vector field $K_{0}=(t^{2}+r^{2})\partial_{t}+2tr\partial_{r}$. More precisely, a direct computation yields,
         \begin{equation*}
             \begin{aligned}
                ( -\Box f)(K_{0}f+2tf)
                &=
                \frac{1}{2}\partial_{t}\left[
                (t^{2}+r^{2})|\partial f|^{2}+4tr(\pt f)(\pr f)
                \right]\\
                &+\pt \left[
                2tf\pt f-f^{2}+{\rm{div}}(xf^{2})
                \right]
                -{\rm{div}}\left[
                2tf\nabla f+\pt (xf^{2})
                \right]\\
                &+{\rm{div}}\left[
                tx(-2(\pt f)^{2}+|\partial f|^{2})
                -(t^{2}+r^{2})(\pt f)\nabla f-2tr (\pr f)\nabla f
                \right].
             \end{aligned}
         \end{equation*}
         Integrating the above identity over $\mathcal{H}_{[1,s]}$, and then using~\eqref{equ:normal}, we have 
         \begin{equation*}
             \mathcal{E}_{\rm{con}}(s,f)-\mathcal{E}_{\rm{con}}(1,f)=-2\int_{1}^{s}\int_{\mathcal{H}_{\tau}}(\tau/t)(K_{0}f+2tf)F\d x,
         \end{equation*}
          which, by differentiating in $s$ and applying the Cauchy-Schwarz inequality, leads us to
          \begin{equation*}
         {{\mathrm{d}\over \mathrm{d}s}}\mathcal{E}_{\rm{con}}(s,f)^{1\over 2}
             \leq \|sF\|_{L^{2}{(\mathcal{H}_{s})}}\Longrightarrow 
             \mathcal{E}_{\rm{con}}(s,f)^{\frac{1}{2}} \lesssim
     \mathcal{E}_{\rm{con}}(1,f)^{\frac{1}{2}}+\int_{1}^{s}\tau\|F\|_{L^{2}(\mathcal{H}_{\tau})}\mathrm{d} \tau.
         \end{equation*}  
         The proof for the estimate of the conformal energy $\mathcal{E}_{\rm{con}}$ is complete.
     \end{proof}

Last, we introduce the following $L^{\infty}-L^{\infty}$ estimates for 3D linear wave equations. The proof is the same as in~\cite[Proposition 3.1]{AlinhacIndiana}, but we give a sketch for the sake of completeness and for the reader’s convenience. 
 \begin{lemma}\label{le:Linfty}
 The following estimates for the 3D linear wave equations hold.
 \begin{enumerate}
     \item \emph{({Contribution from the initial data}).}
     Let $f=f(t,x)$ be the solution to the Cauchy problem with smooth data
  \begin{equation}\label{equ:waveinitial}
     -\Box f=0\quad \mbox{with}\ \ (f,\pt f)_{|\mathcal{H}_{1}}=(f_{0},f_{1}).
 \end{equation}
Assume that the initial data $(f_{0},f_{1})$ satisfy
\begin{equation}\label{est:smallinitial}
    |f_{0}|+\langle r\rangle\left(|\nabla f_{0}|+|f_{1}|\right)\lesssim {}\langle  r\rangle^{-\frac{7}{2}} \ \ \mbox{and}\ \ 
    \sum_{|I|\le 2}\left\|L^{I}\partial^{\leq 1} f_0\right\|\lesssim 1.
\end{equation}
 Then we have 
 \begin{equation*}
     |f(t,x)|\lesssim (1+u)^{-\frac{3}{2}}(1+v)^{-1}.
 \end{equation*}

\item \emph{({Contribution from the source term}).}
Let $f=f(t,x)$ be the solution to the Cauchy problem
	\begin{equation}\label{equ:wavesource}
	-\Box f=F\quad \mbox{with}\quad (f,\pt f)_{|\mathcal{H}_{1}}=(0,0).
    \end{equation}
	Assume that $F$ is supported in $\mathcal{H}_{[1,+\infty)}$ and satisfies
	\begin{equation*}
	\qquad |F|\le C_{F} (1+u)^{-\mu}(1+v)^{-\nu} \ln^{\gamma} (1+v),\ \  (\mu, \nu, \gamma)\in \mathbb{R}_{+}\times \mathbb{R}_+\times \left\{0,1\right\}.
	\end{equation*}
	Define $F_{\mu}(s)=1,\ln (2+s), (1+s)^{1-\mu}/(1-\mu)$ according to $\mu >1,=1,<1$, respectively, then the following pointwise estimates are true.
	\begin{enumerate}

     \item [\emph{(a)}]  Let $(\mu,\nu,\gamma)\in (1,+\infty)\times \left\{3\right\}\times \left\{1\right\}$. Then we have 
		\begin{equation*}
		|f(t,x)|\lesssim C_{F}(1+u)^{-1}(1+v)^{-1}\ln (2+u).
		\end{equation*}
		
		\item [\emph{(b)}]
        Let $(\mu,\nu,\gamma)\in \R_{+}\times (0,2)\times\left\{0\right\}$.
        Then we have 
		\begin{equation*}
		|f(t,x)|\lesssim C_{F}(1+v)^{-\nu+1}F_{\mu}(u).
		\end{equation*}
		
		\item [\emph{(c)}] 
        Let $(\mu,\nu,\gamma)\in \R_{+}\times \left\{2\right\}\times \left\{0\right\}$.
        Then we have 
		\begin{equation*}
		|f(t,x)|\lesssim C_{F}(1+v)^{-1}\ln (1+v)F_{\mu}(u).
		\end{equation*}
		
		\item [\emph{(d)}]
        Let $(\mu,\nu,\gamma)\in \mathbb{R}_{+}\times (2,+\infty)\times \left\{0\right\}$.
        Then we have 
		\begin{equation*}
		|f(t,x)|\lesssim C_{F}(1+u)^{-(\nu-2)}(1+v)^{-1}F_{\mu}(u).
		\end{equation*}
\end{enumerate}
 
 \end{enumerate}
 \end{lemma}

 \begin{proof}
 Proof of (i).
 For any $x\in \R$, we denote 
 \begin{equation*}
     f_{2}(x)=\frac{3+2r^{2}}{\sqrt{1+r^{2}}}f_{0}(x)+2(1+r^{2})^{\frac{1}{2}}x\cdot \nabla f_{0}(x).
 \end{equation*}
  By an elementary computation, we decompose $f=g_{1}+g_{2}+\pt g_{3}$ where
  \begin{equation*}
  \left\{
  \begin{aligned}
      -\Box g_{1}&=0\quad \mbox{with}\ \  (g_{1},\pt g_{1})_{|\mathcal{H}_{1}}=(0,f_{1}),\\
       -\Box g_{2}&=0\quad \mbox{with}\ \  (g_{2},\pt g_{2})_{|\mathcal{H}_{1}}=(0,f_{2}),\\
        -\Box g_{3}&=0\quad \mbox{with}\ \  (g_{3},\pt g_{3})_{|\mathcal{H}_{1}}=(0,f_{0}).
      \end{aligned}
      \right.
  \end{equation*}

  \textbf{Step 1.} Estimates on $g_{1}$ and $g_{2}$.
  Without loss of generality, we assume that the initial data $f_{1}$ is a radial function\footnote{Here, we use the fact that the comparison theorem holds for the 3D wave equations in the hyperbolic setting, which is analogous to the usual flat setting; details can be found in \cite[equations (2.3) and (2.5)]{WaZh25} for instance.}. 
  We denote $(X_{1},F_{1})=(rg_{1},rf_{1})$. From~\eqref{equ:waveinitial}, we directly have 
 \begin{equation*}
     UV X_{1}=0\quad \mbox{with}\ \ (X_{1},\partial_{t}X_{1})_{|\mathcal{H}_{1}}=(0,F_{1}).
 \end{equation*}
 Integrating the above equation from $(v^{-1},v)$ which belongs to the initial hypersurface $\mathcal{H}_{1}$, 
and then using the initial data condition~\eqref{est:smallinitial},  we deduce that 
 \begin{equation}\label{est:VX}
     \left|VX_{1}(u,v)\right|= |VX_{1}(v^{-1},v)|\lesssim \langle v\rangle^{-\frac{5}{2}}.
 \end{equation}

For the case of $u\in (0,1)$, we integrate the above estimate from $(u,u^{-1})$ which belongs to the initial hypersurface, and thus, we obtain, using the initial data condition ~\eqref{est:smallinitial},
\begin{equation*}
    \left|(rg_{1})(u,v)\right|\lesssim 
    \left|(rg_{1})(u,u^{-1})\right|
    +\int_{u^{-1}}^{v}\langle \tau \rangle^{-\frac{5}{2}}\d \tau\lesssim 
    \langle u^{-1}\rangle^{-\frac{3}{2}}.
\end{equation*}

For the case of $u\in [1,+\infty)$, we integrate the estimate~\eqref{est:VX} from $(u,u)$ which belongs to the line $r=0$, and thus, we obtain 
\begin{equation*}
    \left|(rg_{1})(u,v)\right|\lesssim 
    \int_{u}^{v}\langle \tau \rangle^{-\frac{5}{2}}\d \tau\lesssim 
    \langle u\rangle^{-\frac{3}{2}}-\langle v\rangle^{-\frac{3}{2}}.
\end{equation*}
Combining the above two estimates, we conclude that 
\begin{equation}\label{est:g1}
    |g_{1}(u,v)|\lesssim |v-u|^{-1}|(rg_{1})(u,v)|\lesssim (1+u)^{-\frac{3}{2}}(1+v)^{-1}.
\end{equation}
Based on a similar argument as in the above and then using the initial data condition~\eqref{est:smallinitial}, we also conclude that 
\begin{equation}\label{est:g2}
    |g_{2}(u,v)|\lesssim |v-u|^{-1}|(rg_{2})(u,v)|\lesssim (1+u)^{-\frac{3}{2}}(1+v)^{-1}.
\end{equation}

\textbf{Step 2.} Estimate on $g_{3}$.
Using~\eqref{est:L2energy},~\eqref{est:L2conenergy},~\eqref{est:smallinitial} and Lemma~\ref{le:energy}, 
\begin{equation*}
    \sum_{|I|\le 2}\left\|(s/t)L^{I}U\partial g_{3}\right\|_{L^{2}(\mathcal{H}_{s})}+\sum_{|I|\le 2}\left\|s(s/t)^{2}L^{I}\partial g_{3}\right\|_{L^{2}(\mathcal{H}_{s})}\lesssim 1.
\end{equation*}
It then follows from Lemma~\ref{le:Klain} that 
\begin{equation}\label{est:Uptg3}
    |U\partial_{t} g_{3}|\lesssim u^{-\frac{1}{2}}v^{-1}\quad \mbox{and}\quad 
    |\partial_{t} g_{3}|\lesssim u^{-\frac{3}{2}}v^{-1}.
\end{equation}

For the case of $u\in (0,1)$, we integrate the first estimate in \eqref{est:Uptg3} from $(v^{-1},v)$ which belongs to the initial hypersurface, and thus, we obtain 
\begin{equation*}
   \left| \partial_{t}g_{3}(u,v)\right|\lesssim |\partial_{t}g_{3}(v^{-1},v)|+v^{-1}\int_{v^{-1}}^{u}\tau^{-\frac{1}{2}}\d \tau \lesssim v^{-1}.
\end{equation*}
Combining the above estimate with the second one in~\eqref{est:Uptg3}, we deduce that 
\begin{equation}\label{est:g3}
    |\partial_{t}g_{3}(u,v)|\lesssim
    v^{-1}\left(\textbf{1}_{\{0<u<1\}}+u^{-\frac{3}{2}}\textbf{1}_{\{u\ge 1\}}\right)
    \lesssim
    (1+u)^{-\frac{3}{2}}(1+v)^{-1}.
\end{equation}

\textbf{Step 3.} Conclusion.
Combining~\eqref{est:g1},~\eqref{est:g2},~\eqref{est:g3}, and in view that $f=g_{1}+g_{2}+\partial_{t}g_{3}$, we complete the proof of the estimate for $f$.

\smallskip
Proof of (ii). Similarly to the above, we may rely on the comparison theorem and assume without loss of generality that the source term $F$ is a radial function. We denote $(X,G)=(rf,rF)$. From~\eqref{equ:wavesource}, we directly have 
\begin{equation*}
    UVX=G\quad \mbox{with}\ \ (X,\partial_{t}X)_{|\mathcal{H}_{1}}=(0,0).
\end{equation*}
 Let $(\mu,\nu,\gamma)\in (1,+\infty)\times \left\{3\right\}\times \left\{1\right\}$. Integrating the above equation from $(v^{-1},v)$ which belongs to the initial hypersurface $\mathcal{H}_{1}$, we deduce that 
 \begin{equation}\label{est:VX1}
     \left|VX(u,v)\right|\lesssim \int_{v^{-1}}^{u}
     \frac{(v-\tau)\ln(1+v)}{(1+\tau)^{\mu}(1+v)^{3}}\d \tau \lesssim (1+v)^{-2}\ln (1+v).
 \end{equation}

 For the case of $u\in (0,1)$,
we integrate the above estimate from $(u,u^{-1})$ which belongs to the initial hypersurface, and thus, we obtain 
\begin{equation*}
    \left|(rf)(u,v)\right|\lesssim 
    \int_{u^{-1}}^{v}(1+\tau)^{-2}\ln (1+\tau)\d \tau \lesssim (1+u^{-1})^{-1}\ln(1+u^{-1}).
\end{equation*}
For the case of $u\in [1,+\infty)$, we integrate the estimate~\eqref{est:VX1} from $(u,u)$ which belongs to the line $r=0$, and thus, we obtain 
\begin{equation*}
\begin{aligned}
    \left|(rf)(u,v)\right|
    &\lesssim 
    \int_{u}^{v}(1+\tau)^{-2}\ln (1+\tau)\d \tau \\
    &\lesssim 
    (1+v)^{-1}
    \left(r(1+u)^{-1}\ln (2+u)+(\ln(1+u)-\ln(1+v))\right),
    \end{aligned}
\end{equation*}
which completes the proof of the estimate for $f$ in (a).

The proofs of the estimates in (b)--(d) are based on a similar argument (see~\cite[Proposition 3.1]{AlinhacIndiana} for more detail), and we omit it.
\end{proof}

 \subsection {Global existence and almost sharp decay}
   
   In this subsection, we show the global existence and almost sharp decay for the small data solutions to~\eqref{equ:main}. Our statement on the global existence and almost sharp decay estimates is contained in the following Proposition.
\begin{proposition}\label{prop:existence}
 Let $N\in \mathbb{N}^{+}$ with $N\ge 6$.
 Let $(I_{1},I_{2})\in \mathbb{N}^{11}$ with $|I_{1}|\le N-3$ and $|I_{2}|\le N-4$.
 There exists an $\epsilon_{0}>0$ such that for all $\epsilon\in (0,\epsilon_{0})$ and all initial data $(\phi_{0},\phi_1,\psi_{0},\psi_1)$ on $\mathcal{H}_{1}$ satisfying the smallness condition~\eqref{est:smallness1}, 
 the Cauchy problem~\eqref{equ:main} admits a global-in-time solution $(\phi,\psi)$ which enjoys the following almost sharp decay estimates 
 \begin{equation*}
 \begin{aligned}
\left|\Gamma^{I_{1}}\phi\right|&\lesssim_{N} \epsilon (1+v)^{-1}\ln(1+v), \ \
\left|\partial\Gamma^{I_{2}}\phi\right|\lesssim_{N} \epsilon (1+v)^{-1}(2+u)^{-1}\ln (1+v),\\
\left|V\Gamma^{I_{2}}\phi\right|&\lesssim_{N}\epsilon (1+v)^{-2}\ln (1+v),\ \
\left|U\Gamma^{I_{2}}\phi\right|\lesssim_{N} \epsilon (1+v)^{-1}(2+u)^{-1}\ln (1+v),\\
\left|\Gamma^{I_{1}}\psi\right|&\lesssim_{N} \epsilon  (1+v)^{-1}(2+u)^{-1}\left(\ln (2+u) + ((2+u)/(1+v))\ln^{2} (1+v)\right),\\
\left|\partial\Gamma^{I_{2}}\psi\right|&\lesssim_{N} \epsilon  (1+v)^{-1}(2+u)^{-2}\left(\ln (2+u) + ((2+u)/(1+v))\ln^{2} (1+v)\right),\\
\left|V\Gamma^{I_{2}}\psi\right|&\lesssim_{N} \epsilon  (1+v)^{-2}(2+u)^{-1}\left(\ln (2+u) + ((2+u)/(1+v))\ln^{2} (1+v)\right).
     \end{aligned}
 \end{equation*}
\end{proposition}

\begin{proof}
 The proof of Proposition~\ref{prop:existence} relies on a bootstrap argument for a high-order energy of the solution $(\phi,\psi)$. From the definition of $\mathcal{E}(s,f)$, the smallness condition~\eqref{est:smallness1} implies the smallness condition for a high-order energy over the initial hypersurface $\mathcal{H}_{1}$:
\begin{equation}\label{est:smallness}
     \sum_{|I|\le N}\mathcal{E}(1,\Gamma^{I}\phi)^{\frac{1}{2}}+\sum_{|I|\le N}\mathcal{E}(1,\Gamma^{I}\psi)^{\frac{1}{2}}
     \lesssim \epsilon.
\end{equation}
In addition, we consider the following bootstrap assumption for the high-order energy over $\mathcal{H}_{s}$: for $C\gg 1$ and $0<\epsilon\ll C^{-1}$ to be chosen later, 
 \begin{equation}\label{est:Boot}
     (\ln(1+s))^{-1}\sum_{|I|\le N}\mathcal{E}(s,\Gamma^{I}\phi)^{\frac{1}{2}}+\sum_{|I|\le N}\mathcal{E}(s,\Gamma^{I}\psi)^{\frac{1}{2}}
     \leq C\epsilon.
 \end{equation}
 For all initial data $(\phi_{0},\phi_1,\psi_{0},\psi_1)$ satisfying~\eqref{est:smallness1}, we set 
\begin{equation*}
S_{*}=S_{*}(\phi_{0},\phi_1,\psi_{0},\psi_1)=\sup \left\{s\in [1,+\infty):(\phi,\psi) \ \mbox{satisfies}~\eqref{est:Boot}\ \mbox{on}\ [1,S_{*})\right\}>1.
\end{equation*}

We are in a position to complete the proof of 
Proposition~\ref{prop:existence}. Note that, in the following discussions, the implicit constants in $\lesssim$ can depend on the constant $N\in \mathbb{N}^{+}$.

    First of all, from~\eqref{est:L2energy}, Lemma~\ref{le:Klain}, the bootstrap assumption \eqref{est:Boot}, and the definition of the energy $\mathcal{E}(s,f)$,  we deduce, for all $s\in [1,S_{*})$, 
\begin{equation}\label{est:Bootpoint}
    \begin{aligned}
       \sum_{|I|\le N-2}
       \left(\left|\partial \Gamma^{I}\psi\right|+s^{-1}\left|L\Gamma^{I}\psi\right|\right)
      & \lesssim C\epsilon t^{-\frac{1}{2}}s^{-1},\\
      (\ln(1+s))^{-1}
      \sum_{|I|\le N-2}
       \left(\left|\partial \Gamma^{I}\phi\right|+s^{-1}\left|L\Gamma^{I}\phi\right|\right)
      & \lesssim C\epsilon t^{-\frac{1}{2}}s^{-1}.
       \end{aligned}
    \end{equation}
    On the other hand, from~\eqref{equ:main}, for any $I\in \mathbb{N}^{11}$ with $|I|\le N$, there exist constants $\left\{a_{J}^{I}\right\}_{J=0}^{I}$  (independent of $(\phi,\psi)$) such that
\begin{equation*}
-\Box{\Gamma}^{I}\phi=a_{J}^{I}{\Gamma}^{J}\left(\partial_{t}\psi\right)^{2} \ \ \mbox{and} \ \
-\Box {\Gamma}^{I}\psi={ a}_{J}^{I}{\Gamma}^{J}Q_{0}(\phi,\phi).
\end{equation*}
Here, we have used the following commutator relations: 
\begin{equation*}
[-\Box, \partial]=[-\Box, L]=[-\Box, \Omega]=[-\Box,S]+2\Box=0.
\end{equation*}
The above identities will be frequently used in this proof. 

\smallskip
{\bf Step 1.} Global existence. We start with the proof of $S_{*}=+\infty$. 
For any initial data $(\phi_{0},\phi_1,{\psi}_{0},\psi_1)$ satisfying the smallness condition~\eqref{est:smallness1}, we consider the corresponding solution $(\phi,\psi)$ of~\eqref{equ:main}. In what follows, we prove the global existence of $(\phi,\psi)$ by improving the estimates of $(\phi,\psi)$ in the bootstrap assumption~\eqref{est:Boot}.

First, from the smallness condition~\eqref{est:smallness1} and Lemma~\ref{le:Klain}, we directly have 
\begin{equation}\label{est:initsmall}
\begin{aligned}
    \sum_{|I|\le N-3}\left(|\Gamma \Gamma^{I}\phi|+|\Gamma\Gamma^{I}\psi|\right)_{|\mathcal{H}_{1}}\lesssim \epsilon\langle r\rangle^{-\frac{5}{2}},\\
    \sum_{|I|\le N-3}
    \left(|\partial\Gamma \Gamma^{I}\phi|+|\partial\Gamma\Gamma^{I}\psi|\right)_{|\mathcal{H}_{1}}
    \lesssim \epsilon\langle r\rangle^{-\frac{7}{2}}.
    \end{aligned}
\end{equation}
For the case of $u\in (0,1)$, from~\eqref{est:Bootpoint} and the definition of $s$, 
\begin{equation*}
   \sum_{|I|\le N-2}
       \left|U\Gamma^{I}\psi\right|\lesssim  \sum_{|I|\le N-2}
       \left|\partial \Gamma^{I}\psi\right|\lesssim C\epsilon  v^{-1}u^{-\frac{1}{2}}.
\end{equation*}
Integrating the above estimate from $(v^{-1},v,\omega)$ which belongs to the initial hypersurface $\mathcal{H}_{1}$, and then using~\eqref{est:initsmall}, we obtain
\begin{equation}\label{est:localu}
    \sum_{|I|\le N-3}\left|\partial \Gamma^{I}\psi\right|\lesssim \epsilon\langle v-v^{-1}\rangle^{-\frac{5}{2}}+
    C\epsilon v^{-1}\int_{v^{-1}}^{u}\tau^{-\frac{1}{2}}\d \tau \lesssim 
    C\epsilon v^{-1}u^{\frac{1}{2}},
\end{equation}
which directly implies 
\begin{equation*}
    \sum_{|I|\le N-3}\left|t\partial \Gamma^{I}\psi\right|\lesssim C\epsilon tv^{-1}u^{\frac{1}{2}}\lesssim  C\epsilon \Longrightarrow 
    \sum_{|I|\le N-3}\|(t/s)\partial\Gamma^{I}\psi\|_{L^{\infty}(\mathcal{H}_{s})}\lesssim C\epsilon s^{-1}.
\end{equation*}
For the case of $u\in (1,+\infty)$, using again~\eqref{est:Bootpoint} and the definition of $s$, we infer
\begin{equation*}
    \sum_{|I|\le N-3}\left|t\partial \Gamma^{I}\psi\right|\lesssim C\epsilon t^{\frac{1}{2}}s^{-1}\lesssim  C\epsilon \Longrightarrow 
    \sum_{|I|\le N-3}\|(t/s)\partial\Gamma^{I}\psi\|_{L^{\infty}(\mathcal{H}_{s})}\lesssim C\epsilon s^{-1}.
\end{equation*}
Combining the above estimates with~\eqref{est:L2energy},~\eqref{est:Boot} and $N\ge 6$, we have 
\begin{equation*}
\begin{aligned}
    \sum_{|I|\le N}\|\Gamma^{I}(\pt \psi)^{2}\|_{L^{2}(\mathcal{H}_{s})}
    &\lesssim \sum_{\substack{|I_{1}|\le N-3\\ |I_{2}|\le N}}
    \|(t/s)\partial\Gamma^{I_{1}}\psi\|_{L^{\infty}(\mathcal{H}_{s})}\|(s/t)\partial\Gamma^{I_{2}}\psi\|_{L^{2}(\mathcal{H}_{s})}\\
    &\lesssim \sum_{\substack{|I_{1}|\le N-2\\ |I_{2}|\le N}}
    \|(t/s)\partial\Gamma^{I_{1}}\psi\|_{L^{\infty}(\mathcal{H}_{s})}\mathcal{E}(s,\Gamma^{I_{2}}\psi)\lesssim C^{2}\epsilon^{2}s^{-1}.
    \end{aligned}
\end{equation*}
Integrating the above estimate over $[1,s]$ and then using Lemma~\ref{le:energy}, we obtain
\begin{equation*}
    \sum_{|I|\le N}\mathcal{E}(s,\Gamma^{I}\phi)^{\frac{1}{2}}
    \lesssim  \epsilon
    +\sum_{|I|\le N}\int_{1}^{s}\|\Gamma^{I}(\pt \psi)^{2}\|_{L^{2}(\mathcal{H}_{\tau})}\d \tau
    \lesssim
    \epsilon+C^{2}\epsilon^{2}\ln (1+s).
\end{equation*}
This estimate strictly improves the energy estimates of $\phi$ in the bootstrap assumption~\eqref{est:Boot} for $C$ large enough and $\epsilon$ small enough.

\smallskip
Second, using again $N\ge 6$, we have, for any $s\in [1,S_{*})$,  
\begin{equation*}
     \sum_{|I|\le N}\|\Gamma^{I}Q_{0}(\phi,\phi)\|_{L^{2}(\mathcal{H}_{s})}
    \lesssim \sum_{\substack{|I_{1}|\le N-3\\ |I_{2}|\le N}}
    \|Q_{0}(\Gamma^{I_{1}}\phi,\Gamma^{I_{2}}\phi)\|_{L^{2}(\mathcal{H}_{s})}.
\end{equation*}
Based on the above estimate and Lemma~\ref{le:null}, we obtain 
\begin{equation*}
\begin{aligned}
     \sum_{|I|\le N}\|\Gamma^{I}Q_{0}(\phi,\phi)\|_{L^{2}(\mathcal{H}_{s})}
    &\lesssim \sum_{\substack{|I_{1}|\le N-3\\ |I_{2}|\le N}}
    \|\partial \Gamma^{I_{1}}\phi\|_{L^{\infty}(\mathcal{H}_{s})}
    \|t^{-1}L\Gamma^{I_{2}}\phi\|_{L^{2}(\mathcal{H}_{s})}\\
    &+\sum_{\substack{|I_{1}|\le N-3\\ |I_{2}|\le N}}
    \|s^{-1}L\Gamma^{I_{1}}\phi\|_{L^{\infty}(\mathcal{H}_{s})}
    \|(s/t)\partial \Gamma^{I_{2}}\phi\|_{L^{2}(\mathcal{H}_{s})}\\
    &+\sum_{\substack{|I_{1}|\le N-3\\ |I_{2}|\le N}}
    \|(s/t)\partial_{t}\Gamma^{I_{1}}\phi\|_{L^{\infty}(\mathcal{H}_{s})}\|(s/t)\partial_{t}\Gamma^{I_{2}}\phi\|_{L^{2}(\mathcal{H}_{s})},
\end{aligned}
\end{equation*}
which, combined with the estimates ~\eqref{est:L2energy},~\eqref{est:Boot} and~\eqref{est:Bootpoint}, implies
\begin{equation*}
    \sum_{|I|\le N}\|\Gamma^{I}Q_{0}(\phi,\phi)\|_{L^{2}(\mathcal{H}_{s})}\lesssim C^{2}\epsilon^{2}s^{-\frac{3}{2}}
    \ln^{2} (1+s)\lesssim C^{2}\epsilon^{2}s^{-\frac{5}{4}}.
\end{equation*}
Integrating the above estimate over $[1,s]$, and then using Lemma~\ref{le:energy}, we deduce
\begin{equation*}
    \sum_{|I|\le N}\mathcal{E}(s,\Gamma^{I}\psi)^{\frac{1}{2}}
    \lesssim  \epsilon
    +\sum_{|I|\le N}\int_{s_{0}}^{s}\|\Gamma^{I}Q_{0}(\phi,\phi)\|_{L^{2}(\mathcal{H}_{\tau})}\d \tau
    \lesssim
    \epsilon+C^{2}\epsilon^{2}.
\end{equation*}
This estimate strictly improves the energy estimates of $\psi$ in the bootstrap assumption~\eqref{est:Boot} for  $C$ large enough and $\epsilon$ small enough.

\smallskip
We have so far strictly improved the estimates of $(\phi,\psi)$ in the bootstrap assumption~\eqref{est:Boot} and, thus, for all initial data $(\phi_{0},\phi_1, {\psi}_{0},\psi_1)$ satisfying~\eqref{est:smallness1}, {$S_{*}(\phi_{0},\psi_{0})=+\infty$ and} the proof of global existence is complete.

\smallskip
\textbf{Step 2.} Almost sharp decay estimates.
In what follows, we tacitly write the decay rates of the solution pair in terms of $(u,v)$,
while by the local results (see, for example, the proof of~\eqref{est:localu}), the same decay estimates remain valid if one replaces $(u,v)$ by $(2+u,1+v)$.

\smallskip
From~\eqref{est:Bootpoint}, for any $J\in \mathbb{N}^{11}$ with $|J|\le N-2$, we directly have 
\begin{equation*}
    \left|\Gamma^{J}(\pt \psi )^{2}\right|
    \lesssim \sum_{\substack{|J_{1}|\le N-2\\ |J_{2}|\le N-2}}
    \left|\partial \Gamma^{J_{1}}\psi\right|
    \left|\partial \Gamma^{J_{2}}\psi\right|\lesssim \epsilon t^{-1}s^{-2}\lesssim \epsilon v^{-2}u^{-1}.
\end{equation*}
Therefore, from Lemma~\ref{le:Linfty},  we obtain,   for any $J\in \mathbb{N}^{+}$ with $|J|\le N-2$,
\begin{equation*}
    \left| \Gamma^{J}\phi\right|\lesssim 
    \epsilon v^{-1}u^{-1}+\epsilon v^{-1}(\ln v) (\ln u)\lesssim \epsilon v^{-1}\ln^{2}v.
\end{equation*}
Here, we have used (i) of Lemma~\ref{le:Linfty} and the smallness condition~\eqref{est:smallness1} for initial data.
Combining the above estimate with~\eqref{est:t-rpoint}, we obtain a rough decay estimate for $\phi$ with an additional $\ln v$ growth.

\smallskip
Next, from the above decay estimates for $\phi$ and Lemma~\ref{le:null}, for any $J\in \mathbb{N}^{11}$ with $|J|\le N-3$, we directly have
\begin{equation*}
\begin{aligned}
     |\Gamma^{J}Q_{0}(\phi,\phi)|
    &\lesssim t^{-1}\sum_{\substack{|J_{1}|\le N-3\\ |J_{2}|\le N-3}}
    |\partial \Gamma^{J_{1}}\phi|
    |L\Gamma^{J_{2}}\phi|\\
    &+t^{-1}\sum_{\substack{|J_{3}|\le N-3\\ |J_{4}|\le N-3}}
    |L\Gamma^{J_{3}}\phi|
    |\partial\Gamma^{J_{4}}\phi|\\
    &+(s/t)^{2}\sum_{\substack{|J_{5}|\le N-3\\ |J_{6}|\le N-3}}
    |\partial\Gamma^{J_{5}}\phi|
    |\partial\Gamma^{J_{6}}\phi|\lesssim \epsilon v^{-3}u^{-1}\ln^{4}v.
\end{aligned}
\end{equation*}
Therefore, from Lemma~\ref{le:Linfty},  we infer, for any $J\in \mathbb{N}^{11}$ with $|J|\le N-3$,
\begin{equation*}
    \left| \Gamma^{J}\psi\right|\lesssim \epsilon v^{-1}u^{-1}+\epsilon v^{-1}u^{-1+\delta_{1}}\lesssim \epsilon v^{-1}u^{-1+\delta_{1}},\quad \mbox{for any}\ \ 0<\delta_{1}\ll 1.
\end{equation*}
Here, we have used again (i) of Lemma~\ref{le:Linfty} and the smallness condition~\eqref{est:smallness1} for initial data.
Combining the above estimate with~\eqref{est:t-rpoint}, we obtain a rough decay estimate for $\psi$ with an additional $u^{\delta_{1}}$ growth.

\smallskip
Then, from the above estimate for $\psi$ and~\eqref{est:t-rpoint}, for any $J\in \mathbb{N}^{+}$ with $|J|\le N-3$, 
\begin{equation*}
    \left| \Gamma^{J}(\pt \psi )^{2}\right|
    \lesssim \sum_{\substack{|J_{1}|\le N-3\\ |J_{2}|\le N-4}}
    \left|\partial \Gamma^{J_{1}}\psi\right|
    \left|\partial \Gamma^{J_{2}}\psi\right|
    \lesssim \epsilon v^{-2}u^{-\frac{5}{2}+\delta_{1}}.
\end{equation*}
It follows from Lemma~\ref{le:Linfty} that 
\begin{equation*}
    \left|\Gamma^{J}\phi\right|
    \lesssim 
    \epsilon v^{-1}u^{-1}+
    \epsilon v^{-1}\ln^{}v
    \lesssim \epsilon v^{-1}\ln v,\quad 
    \mbox{for any} \ \ |J|\le N-3.
\end{equation*}
Combining the above estimate with~\eqref{est:t-rpoint}, we complete the proof of estimates for $\phi$.

\smallskip
Last, recall from \eqref{equ:main-new} that the scalar $\bar{\psi} = \psi + \frac{1}{2}\phi^2$ satisfies $-\Box \bar{\psi}=\phi(\pt \psi)^{2}$.
Using the above estimates for $\phi$ and $\psi$, for $J\in \mathbb{N}^{+}$ with $|J|\le N-3$, we have 
\begin{equation*}
   \left|  \Gamma^{J}\left(\phi(\pt \psi)^{2}\right)\right|\lesssim
   \epsilon^3 u^{-\frac{5}{2}+\delta_1}v^{-3}\ln v,\quad \mbox{for any}\ 0<\delta_{1}\ll 1.
\end{equation*}
It follows from Lemma~\ref{le:Linfty} that 
\begin{equation*}
    \big| \Gamma^{J}\bar{\psi}\big|\lesssim \epsilon u^{-1}v^{-1}\ln u\Longrightarrow 
    \left|\Gamma^{J}\psi\right|
\lesssim \epsilon u^{-1}v^{-1}\ln u + \epsilon v^{-2}\ln^{2} v.
\end{equation*}
Combining the above estimates with~\eqref{est:t-rpoint}, we complete the proof.
\end{proof}

Note that, from~\eqref{est:t-rpoint} and Proposition~\ref{prop:existence}, for any $(i_{1},i_{2},i_{3},i_{4},I)\in \mathbb{N}^{15}$ with $i_{1}+i_{2}+i_{3}+i_{4}+|I|\le N-3$, we infer
\begin{subequations}\label{est:almost1}
\begin{align}
    &\left| \partial_t^{i_{1}} U^{i_{2}}V^{i_{3}} \Deltas^{i_{4}} \Gamma^{I}\phi\right|\lesssim_{N} \epsilon v^{-1-i_{3}}u^{-i_{1}-i_{2}}\ln^{}v,\\
    &\left|\partial_t^{i_{1}} U^{i_{2}}V^{i_{3}} \Deltas^{i_{4}} \Gamma^{I}\psi\right|\lesssim_{N} \epsilon v^{-1-i_{3}}u^{-1-i_{1}-i_{2}}  { \big(\ln u + (u/v)\ln^{2} v\big)}.
    \end{align}
\end{subequations}
Using again $-\Box \bar{\psi} = \phi (\partial_t \psi)^2$ and~\eqref{est:t-rpoint}, we also find
\begin{equation}\label{est:almost2}
		\big|\partial_t^{i_{1}}U^{i_2} V^{i_3} \Deltas^{i_4} \Gamma^{I} \bar{\psi}\big|\\
        \lesssim_{N} \epsilon v^{-1-i_3} u^{-1-i_1-i_2} \ln u.
       \end{equation}
The above estimates will be used frequently in the remainder of this article.
  
Furthermore, in view of the above estimates \eqref{est:almost1} and the definitions of $\{\mathfrak{c}_i\}_{i=1}^{4}$ in Theorem \ref{thm:main1}, we directly have
\begin{equation}
\label{eq:roughboundofthepreciseconstants}
|\mathfrak{c}_1|+|\mathfrak{c}_3| \lesssim \ep^2\quad \mbox{and}\quad |\mathfrak{c}_2|+|\mathfrak{c}_4| \lesssim \ep^3.
\end{equation}

\section{Precise decay of the solution}\label{sec:precise-decay}

In this section, we establish the precise decay for the solution to the coupled wave system \eqref{equ:main}. In Section \ref{subsect:alternativeformsofwaveeqs}, we first give some alternative forms of this wave system. Then, the precise decay for $(\phi,\psi,\bar{\psi})$  in Regions $\mathrm{II}$ and $\mathrm{I}$ are studied in Sections \ref{subsect:precise:regionII} and \ref{subsect:precise:regionI}, respectively. Meanwhile, in Section \ref{SS:0order}, we also establish the precise decay for the zeroth-order mode of $(\phi, \bar{\psi})$ globally in the full space-time region for $u\geq 2$.

\smallskip
For further reference, we define the following two different types of interior and exterior space-time regions:
\begin{equation}
\lab{def:differentspacetimeregions}
\begin{aligned}
\mathcal{C}_{\rm{int},\delta}=\big\{r\le \frac{1}{2}u^{1-\delta}\big\}\quad &\mbox{and}\quad 
\mathcal{C}_{\rm{ext},\delta}=\big\{r\ge \frac{1}{2}u^{1-\delta}\big\},\\
\mathcal{D}_{\rm{int},\delta}=\big\{r\le \frac{1}{2}u^{1+\delta}\big\}\quad &\mbox{and}\quad 
\mathcal{D}_{\rm{ext},\delta}=\big\{r\ge \frac{1}{2}u^{1+\delta}\big\}.
\end{aligned}
\end{equation}

\subsection{Alternative forms of the wave system}
\lab{subsect:alternativeformsofwaveeqs}

Recall that, for any two scalar functions $f$ and $g$, the wave equation $-\Box f=g$ can be rewritten as:
\begin{equation}\label{eq:alternativeformsofwaveeq}
UV f-2r^{-1}\partial_{r}f-r^{-2}\Deltas f=g\Longleftrightarrow UV(rf)-r^{-2}\Deltas (rf)=rg.
\end{equation}
Recall also the following scalars constructed from the solution $(\phi,\psi)$:
\begin{equation*}
\bar{\psi}=\psi+\frac{1}{2}\phi^2 \quad \mbox{and}\quad 
(\Phi,\Psi,\bar{\Psi})=(r\phi,r\psi,r\bar{\psi}).
\end{equation*}
Hence, from~\eqref{equ:main} and~\eqref{eq:alternativeformsofwaveeq}, it is easy to check that 
 \begin{equation}\label{equ:PhiPsi:waveUV}
\left\{
\begin{aligned}
    \left(UV-r^{-2}\Deltas\right)\Phi&=r^{-1}\left(\pt \Psi\right)^{2},\\
    \left(UV-r^{-2}\Deltas\right)\Psi&=rQ_0(\phi,\phi).
    \end{aligned}
\right.
 \end{equation}
Then, from~\eqref{eq:alternativeformsofwaveeq} and the definition of $(\Phi,\bar{\Psi})$, we also obtain
 \begin{equation}\label{equ:PhiTPsi1}
\left\{
\begin{aligned}
    \left(UV-r^{-2}\Deltas\right)\Phi&=r^{-1}\left(\pt \Psi\right)^{2},\\
    \left(UV-r^{-2}\Deltas\right)\bar{\Psi}&=r^{-2}\Phi\left(\pt \Psi\right)^{2}.  
    \end{aligned}
\right.
 \end{equation}

\subsection{Estimates in Region $\mathrm{II}$}
\lab{subsect:precise:regionII}


In this section, we establish the asymptotic behavior of the solution $(\phi, \bar{\psi})$ in the space-time Region II. Define 
\begin{equation}
\mathcal{C}_{\delta}:=\left\{r\ge \frac{1}{2}\exp(u^{\delta})\right\},\quad \ \mbox{for}\ \ 0<\delta\ll 1.
\end{equation}
In view of Definition \ref{def:regionsIandII:intro}, we have ${\textrm{II}}=\mathcal{C}_{\delta}$.

\subsubsection{Precise decay of $\phi$}

 We now derive the precise decay for $\phi$ in Region II. We start with the following asymptotic behavior for $V\Phi$ in the space-time region $\mathcal{C}_{\frac{\delta}{2}}$.
 
\begin{lemma}
\label{lem:preciseofVPhiinRegionII}
In the space-time region $\mathcal{C}_{\frac{\delta}{2}}$, we have 
\begin{equation}\label{est:VPhiF}
\left|V\Phi-2\mathfrak{c}_{3}v^{-1}\right|\lesssim \ep v^{-1} u^{-\delta}.
\end{equation}
\end{lemma}
\begin{proof}
Multiplying both sides of the first line of \eqref{equ:PhiTPsi1} by $v$, we obtain 
\begin{equation*}
U\left(vV\Phi\right)=\frac{v}{r^{2}}\Deltas \Phi+\frac{v}{r}\left(\pt \Psi\right)^{2}.
\end{equation*}
Integrating this equation from $(u_{0}(v),v)$ where $u_{0}(v)$ is defined such that
\begin{equation}
\label{def:u_0(v)}
(u_{{0}}(v),v,\omega)\in \mathcal{H}_1\Longleftrightarrow u_{0}(v)=v^{-1},
\end{equation}
we deduce that
\begin{equation}
\label{eq:RegionII:VPhi:integralalongU}
\begin{aligned}
&\left(vV\Phi\right)(u,v)-\left(vV\Phi\right)(v^{-1},v)\\
&=\frac{1}{2}\int_{v^{-1}}^u \bigg(\frac{v}{r^{2}}\Deltas \Phi+\frac{v}{r}\left(\pt \Psi\right)^{2}\bigg)(\sigma, v)
\d \sigma\\
&=2\mathfrak{c}_{3}
-\int_{u}^{+\infty}\left(\pt \Psi\right)^{2}(\sigma,+\infty)\d \sigma
-\int_{0}^{v^{-1}}\left(\pt \Psi\right)^{2}(\sigma,+\infty)\d \sigma\\
&+\frac{1}{2}\int_{v^{-1}}^{u}\left(\frac{v}{r^{2}}\Deltas\Phi\right)(\sigma,v)\d \sigma
-\frac{1}{4}\int_{v}^{+\infty}\int_{v^{-1}}^{u}
V\left[\frac{v}{r}\left(\pt \Psi\right)^{2}\right]\left(\sigma,\nu\right)\d \sigma \d \nu.
\end{aligned}
\end{equation}
Here, we pushed the integral along constant $v$ to null infinity using the Fundamental Theorem of Calculus.

First, by the smallness condition ~\eqref{est:smallness1} of initial data $(\phi_{0},\psi_{0})$, we have
\begin{equation}
\lab{eq:RegionII:VPhi:initialdatacontrol}
  \left|  \left(vV\Phi\right)(v^{-1},v)\right|\lesssim
 \ep v^{-\delta}.
\end{equation}
Second, using $(u,v)\in \mathcal{C}_{\frac{\delta}{2}}$ and the almost sharp decay estimates~\eqref{est:almost1}, we deduce
\begin{equation}
\label{eq:RegionII:VPhi:lowerordercontrol}
\begin{aligned}
\left|\int_{u}^{+\infty}\left(\pt \Psi\right)^{2}(\sigma,+\infty)\d \sigma\right|
\lesssim&\epsilon^{2} \int_{u}^{+\infty} \left(\sigma^{-4} \ln^{2} \sigma\right)\d \sigma\lesssim \epsilon^{2}u^{-3}\ln^{2}u,\\
 \left|\int_{v^{-1}}^{u}\left(\frac{v}{r^{2}}\Deltas\Phi\right)(\sigma,v)\d \sigma\right|\lesssim&\epsilon v^{-2}\ln v\int_{v^{-1}}^{u}(v-\sigma)\d \sigma\lesssim \epsilon uv^{-1}\ln v\lesssim \ep v^{-\delta}.
\end{aligned}
\end{equation}
Then, using again the almost sharp decay estimates~\eqref{est:almost1}, we find 
\begin{equation}\label{eq:RegionII:VPhi:lowerordercontrol2}
    \left|\int_{0}^{v^{-1}}\left(\pt \Psi\right)^{2}(\sigma,+\infty)\d \sigma   \right|
    \lesssim \|\pt \Psi\|_{L^{\infty}}^{2}v^{-1}
\lesssim \ep^2 v^{-1}.
\end{equation}
Last, using $(U,V)=(\pt-\pr,\pt+\pr)$, we compute 
\begin{equation*}
\begin{aligned}
V\left[\frac{v}{r}\left(\pt \Psi\right)^{2}\right]
=-\frac{u}{r^{2}}\left(\pt \Psi\right)^{2}+\frac{v}{r}\left(\pt \Psi\right)\left(UV\Psi\right)
+\frac{v}{r}\left(\pt \Psi\right)\left(VV\Psi\right).
\end{aligned}
\end{equation*}
Therefore, from the wave equation \eqref{equ:PhiPsi:waveUV} and the almost sharp decay estimates~\eqref{est:almost1}, we infer that
\begin{equation*}
\begin{aligned}
\left|\left[V\left(\frac{v}{r}\left(\pt \Psi\right)^{2}\right)\right](u,v)\right|
&\lesssim \left|\pt \Psi\right|\left(\frac{u}{r^{2}}\left|\pt \Psi\right|+\frac{v}{r}\left|VV\Psi\right|\right)\\
&+\left|\pt \Psi\right|\left(v\left|Q_{0}(\phi,\phi)\right|+\frac{v}{r^{3}}\left|\Deltas \Psi\right|\right)\lesssim \ep^2v^{-2}u^{-3}\ln^4 v\ln^{2}u.
\end{aligned}
\end{equation*}
Integrating the above estimate over $ [v^{-1},u]\times
[v,+\infty)$, we obtain
\begin{equation}
\lab{eq:RegionII:VPhi:spacetimeintegralcontrol}
\left|\int_{v}^{+\infty}\int_{v^{-1}}^{u}V\left[\frac{v}{r}\left(\pt \Psi\right)^{2}\right]\left(\sigma,\nu\right)\d \sigma \d \nu\right|\lesssim \ep^{2} v^{-1}\ln^4 v.
\end{equation}
Combining the estimates \eqref{eq:RegionII:VPhi:initialdatacontrol}--\eqref{eq:RegionII:VPhi:spacetimeintegralcontrol} with \eqref{eq:RegionII:VPhi:integralalongU},
 we complete the proof for~\eqref{est:VPhiF}.
\end{proof}

Building on the above precise decay for $V\Phi$, we next compute the precise decay of $\phi$ in the space-time region $\mathcal{C}_{\delta}$.

\begin{proposition}\label{prop:pointphi}
In the space-time region $\mathcal{C}_{\delta}$, we have
\begin{equation}\label{est:pointphi}
\left|\phi(u,v,\omega)-\mathfrak{c}_{3}\phi_{L}(u,v)\right|\lesssim \epsilon u^{-\frac{\delta}{2}} {\phi_{L}}(u,v),
\end{equation}
where the function $\phi_{L}$ is defined as in \eqref{eq:phibarandpsibar}.
\end{proposition}

\begin{proof}
Integrating the estimate~\eqref{est:VPhiF} from $(u,\gamma(u))$ where $\gamma(u)$ is defined such that
\begin{equation*}
(u,\gamma(u))\in\{r=\frac{1}{2}\exp(u^{\frac{\delta}{2}})\}\Longleftrightarrow
\gamma(u)=u+\exp (u^{\frac{\delta}{2}}),
\end{equation*}
we deduce that
\begin{equation*}
\begin{aligned}
(r\phi)(u,v)-(r\phi)(u,\gamma(u))
&=({\mathfrak{c}}_{3}+O(\ep u^{-\delta}))(\ln v-\ln u)\\
&+({\mathfrak{c}}_{3}+O(\ep u^{-\delta}))(\ln u-\ln \gamma(u)).
\end{aligned}
\end{equation*}
In view of the definitions for $(u,\gamma(u))$ and the space-time region $\mathcal{C}_{\delta}$, we find 
\begin{equation*}
    |\ln u-\ln \gamma(u)|\lesssim \ln u+\ln \gamma(u)\lesssim u^{-\frac{\delta}{2}}\ln v.
\end{equation*}
Next, using again the almost sharp decay estimates~\eqref{est:almost1}, we infer
\begin{equation*}
\left|(r\phi)(u,\gamma(u))\right|
\lesssim \epsilon (\gamma(u)-u)\gamma^{-1}(u)\ln \gamma(u)
\lesssim \epsilon \ln \gamma(u)\lesssim \epsilon u^{-\frac{\delta}{2}}\ln v.
\end{equation*}
Combining the above estimates with~\eqref{eq:roughboundofthepreciseconstants}, we complete the proof for~\eqref{est:pointphi}.
\end{proof}

\subsubsection{Precise decay of $\bar{\psi}$}
We now derive the precise decay for $\bar{\psi}$ in Region II. We start with a technical estimate for $V(\Phi/\ln v)$ in the space-time region $\mathcal{C}_{{\rm{ext}},2\delta}$.

\begin{lemma}\label{le:derivativePhiPsi}
In the space-time region $\mathcal{C}_{{\rm{ext}},2\delta}$, we have
    \begin{equation}\label{est:Secondtech}
\left|V\left(\frac{\Phi}{\ln v}\right)\right|\lesssim \frac{\ep \ln^{2} u}{v\ln^{2}v}.
\end{equation}
\end{lemma}

\begin{proof}
Note that, from definition \eqref{def:differentspacetimeregions}, it holds true that
\begin{equation*}
\mathcal{C}_{{\rm{ext}},2\delta}=\big(\mathcal{C}_{\rm{ext},2\delta}
\cap \mathcal{D}_{\rm{int},2\delta}\big)\cup \mathcal{D}_{\rm{ext},2\delta}.
\end{equation*}
We split the proof into the following two parts according to the subregion under consideration. Note that it suffices to consider the case where $u$ is sufficiently large.

\smallskip
\textbf{Step 1.} Estimate in $\mathcal{C}_{\rm{ext},2\delta}
\cap \mathcal{D}_{\rm{int},2\delta}$. 
In view of the almost sharp decay estimates~\eqref{est:almost1} and the definition of $\Phi$, we infer
\begin{equation*}
    \left|V\left(\frac{\Phi}{\ln v}\right)\right|\lesssim 
    \left|\frac{rV\phi}{\ln v}\right|
+\left|\frac{\phi}{\ln v}\right|
    +\left|\frac{2r\phi}{v\ln^{2}v}\right|\lesssim \frac{\epsilon}{v}.
\end{equation*}
Note that, in this space-time region, we directly have 
\begin{equation*}
  u+  u^{1-2\delta}\le v\le u+u^{1+2\delta}\Longrightarrow 
    \ln v\sim \ln u,
\end{equation*}
which completes the proof of the estimate \eqref{est:Secondtech} in this region.

\smallskip
\textbf{Step 2.} Estimate in $\mathcal{D}_{\rm{ext},2\delta}$.  We claim that, in the space-time region $\mathcal{D}_{\rm{ext},2\delta}$, 
\begin{equation}\label{est:Firsttech}
\left|U\Phi-\ln v \left(\pt \Psi\right)^{2}\right|\lesssim \ep u^{-1}\ln u.
\end{equation}
Indeed, by the Fundamental Theorem of Calculus, we deduce that
\begin{equation}
\lab{esti:UPhi-logvptPsi2:Dext2delta:formula}
\begin{aligned}
\left[U\Phi-\ln v\left(\pt \Psi\right)^{2}\right](u,v)
&=\left[U\Phi-\ln v\left(\pt \Psi\right)^{2}\right](u,\gamma(u))\\
&+\frac{1}{2}\int_{\gamma(u)}^{v}
V\left[\left(U\Phi-\ln v\left(\pt \Psi\right)^{2}\right)\right](u,\nu)\d \nu.
\end{aligned}
\end{equation}
Here, $(u,\gamma(u))$ is chosen such that it lies on the curve $r=\frac{1}{2}u^{1+2\delta}$, that is,
\begin{equation*}
\gamma(u)=u+u^{1+2\delta}\Longrightarrow \ln \gamma(u)\sim \ln u.
\end{equation*}
Therefore, using again
the almost sharp decay estimates~\eqref{est:almost1}, we have
\begin{equation}
\lab{esti:UPhi-logvptPsi2:gammau}
    \left|\left[U\Phi-\ln v\left(\pt \Psi\right)^{2}\right](u,\gamma(u))\right|\lesssim \ep u^{-1}\ln u.
\end{equation}
On the other hand, using again the system \eqref{equ:PhiPsi:waveUV}, we compute 
\begin{equation}\label{eq:non-trans}
\begin{aligned}
V\left[U\Phi-\ln v\left(\pt \Psi\right)^{2}\right]
&=\frac{u}{rv}\left(\pt \Psi\right)^{2}
-\frac{\ln v}{r^{2}}(\pt \Psi)(\Deltas\Psi)\\
&+\frac{1}{r^{2}}\Deltas\Phi-\ln v
(\pt \Psi)\left(rQ_{0}(\phi,\phi)+VV\Psi\right).
\end{aligned}
\end{equation}
Therefore, using the almost sharp decay estimates~\eqref{est:almost1}, we infer that, for any $(u,v)\in \mathcal{D}_{\rm{ext},2\delta}$,
\begin{equation*}
\left|V\left[U\Phi-\ln v\left(\pt \Psi\right)^{2}\right](u,v)\right|\lesssim \ep v^{-2}u^{-1}\ln^{5}v+\ep v^{-2}\ln v,
\end{equation*}
which then yields 
\begin{equation*}
\begin{aligned}
&\left|\int_{\gamma(u)}^{v}
V\left[\left(U\Phi-\ln v\left(\pt \Psi\right)^{2}\right)\right](u,\nu)\d \nu\right|\\
&\lesssim
\ep \int_{\gamma(u)}^{v}\left(\nu^{-2}u^{-1}\ln^5 \nu +\nu^{-2}\ln \nu \right)\d \nu\lesssim \ep u^{-1}\ln u.
\end{aligned}
\end{equation*}
Combining the above estimate with \eqref{esti:UPhi-logvptPsi2:Dext2delta:formula} and \eqref{esti:UPhi-logvptPsi2:gammau}, we complete the proof of~\eqref{est:Firsttech}.

\smallskip
Then, integrating along $U$ from $u_{0}(v)$, where $u_0(v)$ is given as in \eqref{def:u_0(v)} and equals $v^{-1}$, 
we deduce that, for any $(u,v)\in \mathcal{D}_{\rm{ext},2\delta}$,
\begin{equation*}
\begin{aligned}
    V\left(\frac{\Phi}{\ln v}\right)(u,v)
    &=V\left(\frac{\Phi}{\ln v}\right)(v^{-1},v)+\frac{1}{2}\int_{v^{-1}}^{u}UV\left(\frac{\Phi}{\ln v}\right)(\sigma,v)\d \sigma.
    \end{aligned}
\end{equation*}
In view of the initial data condition~\eqref{est:smallness1}, it is easy to check that
\begin{equation*}
\left|V\left(\frac{\Phi}{\ln v}\right)(v^{-1},v)\right|\lesssim 
(\ln v)^{-1}\left(|\phi(v^{-1},v)|+|vV\phi(v^{-1},v)|\right)\lesssim
\ep v^{-1}(\ln v)^{-2}.
\end{equation*}
On the other hand, using again \eqref{equ:PhiPsi:waveUV}, we compute
\begin{equation*}
UV\left(\frac{\Phi}{\ln v}\right)
=
\frac{1}{r^{2}}\frac{\Deltas\Phi}{\ln v}
-\frac{2}{v\ln ^{2}v}\left[U\Phi-\ln v\left(\pt \Psi\right)^{2}\right]
+\frac{u}{rv}\left[\frac{\left(\pt \Psi\right)^{2}}{\ln v}\right].
\end{equation*}
Plugging the above identity back into~\eqref{est:almost1} and~\eqref{est:Firsttech}, we obtain
\begin{equation*}
\left|UV\left(\frac{\Phi}{\ln v}\right)(u,v)\right|\lesssim 
\frac{\epsilon\ln u}{vu\ln^2 v}
\Longrightarrow 
\left| \int_{v^{-1}}^{u}UV\left(\frac{\Phi}{\ln v}\right)(\sigma,v)\d \sigma\right|\lesssim \frac{\ep \ln^{2} u}{v\ln^{2}v}.
\end{equation*}
Combining the above estimates, we complete the proof of~\eqref{est:Secondtech} in this region.
\end{proof}

Next, we make use of the above precise decay for $V(\Phi/\ln v)$ to study the asymptotic behavior for $V\bar{\Psi}$ in the space-time region $\mathcal{C}_{\frac{\delta}{2}}$.
\begin{lemma}\label{le:VPsiF}
In the space-time region $\mathcal{C}_{\frac{\delta}{2}}$, we have 
\begin{equation}\label{est:VPsiF}
\left|V\bar{\Psi}-
2\mathfrak{c}_{4}v^{-2}\ln v\right|\lesssim 
\ep v^{-2} u^{-\frac{\delta}{4}}\ln v .
\end{equation}
\end{lemma}

\begin{proof}
By~\eqref{equ:PhiTPsi1}, we compute
\begin{equation*}
U\left(\frac{v^{2}}{\ln v}V\bar{\Psi}\right)
=\left(\frac{v}{r}\right)^{2}\frac{\Deltas \bar{\Psi}}{\ln v}
+\left(\frac{v}{r}\right)^{2}\left(\frac{\Phi}{\ln v}\right)\left(\pt \Psi\right)^{2}.
\end{equation*}
It follows from the Fundamental Theorem of Calculus that
\begin{equation*}
\begin{aligned}
&\left(\frac{v^{2}}{\ln v}V\bar{\Psi}\right)(u,v)-\left(\frac{v^{2}}{\ln v}V\bar{\Psi}\right)(v^{-1},v)\\
&=2\mathfrak{c}_{4}-2\int_{u}^{+\infty}\left[\left(\frac{\Phi}{\ln v}\right)\left(\pt \Psi\right)^{2}\right](\sigma,+\infty)\d \sigma
-
2\int_{0}^{v^{-1}}\left[\left(\frac{\Phi}{\ln v}\right)\left(\pt \Psi\right)^{2}\right](\sigma,+\infty)\d \sigma\\
&-\frac{1}{4}\int_{v}^{+\infty}\int_{v^{-1}}^{u}V\left[\left(\frac{v}{r}\right)^{2}\left(\frac{\Phi}{\ln v}\right)\left(\pt \Psi\right)^{2}\right]\left(\sigma,\nu\right)\d \sigma \d \nu
+\frac{1}{2}\int_{v^{-1}}^{u}
\left(\frac{v}{r}\right)^{2}\frac{\Deltas \bar{\Psi}}{\ln v}(\sigma,v)
\d \sigma.
\end{aligned}
\end{equation*}

First, by the initial data condition~\eqref{est:smallness1}, we have
\begin{equation*}
   \left| \left(\frac{v^{2}}{\ln v}V\bar{\Psi}\right)(v^{-1},v)\right|\lesssim \ep v^{-\delta}.
\end{equation*}
Then, using the almost sharp decay estimates \eqref{est:almost1}, we deduce
\begin{align*}
\left|\int_{u}^{+\infty}\left[\left(\frac{\Phi}{\ln v}\right)\left(\pt \Psi\right)^{2}\right](\sigma,+\infty)\d \sigma\right|
&\lesssim \epsilon \int_{u}^{+\infty} \sigma^{-4+\delta}\d \sigma\lesssim  \ep u^{-3+\delta},\\
\bigg|\int_{0}^{v^{-1}}\left[\left(\frac{\Phi}{\ln v}\right)\left(\pt \Psi\right)^{2}\right](\sigma,+\infty)\d \sigma   \bigg|
&\lesssim v^{-1}
\bigg\|\frac{\Phi}{\ln v}(\pt \Psi)^{2}\bigg\|_{L^{\infty}}
\lesssim
\ep v^{-1}.
\end{align*}
Next, using the almost sharp decay estimates \eqref{est:almost2} for $\bar{\psi}$, we deduce 
\begin{equation*}
    \left|\left(\frac{v}{r}\right)^{2}\frac{\Deltas \bar{\Psi}}{\ln v}(u,v)\right|\lesssim \ep u^{-1}(\ln u)(\ln v)^{-1}\Longrightarrow
\left|\int_{v^{-1}}^{u}
\left(\frac{v}{r}\right)^{2}\frac{\Deltas \bar{\Psi}}{\ln v}(\sigma,v)
\d \sigma\right|\lesssim \ep u^{-\frac{\delta}{4}}.
\end{equation*}
In view of the above estimates, we infer
\begin{equation*}
\begin{aligned}
\left(\frac{v^{2}}{\ln v}V\bar{\Psi}\right)(u,v)
&=2\mathfrak{c}_{4}+O\left(\ep v^{-\delta}+\ep v^{-1}+\ep u^{-\frac{\delta}{4}}+\ep u^{-3+\delta}\right)\\
&-\frac{1}{4}\int_{v}^{+\infty}\int_{v^{-1}}^{u}V\left[\left(\frac{v}{r}\right)^{2}\left(\frac{\Phi}{\ln v}\right)\left(\pt \Psi\right)^{2}\right]\left(\sigma,\nu\right)\d \sigma \d \nu.
\end{aligned}
\end{equation*}
Using the wave equations \eqref{equ:PhiPsi:waveUV}, we check that
\begin{equation*}
\begin{aligned}
&V\left[\left(\frac{v}{r}\right)^{2}\left(\frac{\Phi}{\ln v}\right)\left(\pt \Psi\right)^{2}\right]\\
&=-2\frac{uv}{r^{3}}\left(\frac{\Phi}{\ln v}\right)\left(\pt \Psi\right)^{2}+\left(\frac{v}{r}\right)^{2}\left(V\left(\frac{\Phi}{\ln v}\right)\right)\left(\pt \Psi\right)^{2}\\
&+\left(\frac{v}{r}\right)^{2}\left(\frac{\Phi}{\ln v}\right)\left(\pt \Psi\right)
\left(r^{-2}\Deltas \Psi+rQ_{0}(\phi,\phi)+VV\Psi\right).
\end{aligned}
\end{equation*}
Hence, using Lemma~\ref{le:derivativePhiPsi} and the almost sharp decay estimates~\eqref{est:almost1}, we find 
\begin{equation*}
\begin{aligned}
\left|V\left[\left(\frac{v}{r}\right)^{2}\left(\frac{\Phi}{\ln v}\right)\left(\pt \Psi\right)^{2}\right]\right|
\lesssim \ep v^{-1}(\ln^{-2}v)u^{-4}\ln^2u+\ep v^{-2}u^{-3}\ln^{2}v,
\end{aligned}
\end{equation*}
which directly implies
\begin{equation*}
\left|\int_{v}^{+\infty}\int_{v^{-1}}^{u}V\left[\left(\frac{v}{r}\right)^{2}\left(\frac{\Phi}{\ln v}\right)\left(\pt \Psi\right)^{2}\right]\left(\sigma,\nu\right)\d \sigma \d \nu\right|\lesssim \ep (\ln v)^{-1}\lesssim \ep u^{-\frac{\delta}{4}}.
\end{equation*}
Combining the above estimates, we complete the proof of \eqref{est:VPsiF} in this region.
\end{proof}

Next, we deduce the asymptotic behavior of $\bar{\psi}$ in the space-time region $\mathcal{C}_{\delta}$.

\begin{proposition}\label{prop:psi}
In the space-time region $\mathcal{C}_{\delta}$, we have 
    \begin{equation}\label{est:pointpsi:Cdeltaregion}
        \left|\bar{\psi}(u,v,\omega)-r^{-1}{\Psi}(u,+\infty,\omega)+\mathfrak{c}_{4}r^{-1}v^{-1}\ln v\right|\lesssim \ep r^{-1}v^{-1}u^{-\frac{\delta}{4}}\ln v.
    \end{equation}
\end{proposition}

\begin{proof}
Integrating~\eqref{est:VPsiF} from null infinity, we find 
\begin{align*}
\bar{\Psi}(u,v,\omega)={}&\bar{\Psi}(u,+\infty,\omega) -\frac{1}{2}\int_{v}^{+\infty} V\bar{\Psi} (u,\sigma,\omega) \d\sigma\notag\\
=&\bar{\Psi}(u,+\infty,\omega)- \frac{1}{2}\int_{v}^{+\infty} \Big(2\mathfrak{c}_{4}\sigma^{-2}\ln \sigma +O\left(\ep \sigma^{-2} u^{-\frac{\delta}{4}}\ln \sigma\right)\Big) \d\sigma\notag\\
=&\bar{\Psi}(u,+\infty,\omega) -\mathfrak{c}_{4} v^{-1}\ln v + 
O(\ep v^{-1}u^{-\frac{\de}{4}}\ln v).
\end{align*}
On the other hand, from Proposition~\ref{prop:existence}, we directly have 
\begin{equation*}
    \bar\Psi(u,+\infty,\omega)=\lim_{v\to +\infty}\left(\Psi(u,v,\omega)+\frac{r}{2}\phi^{2}(u,v,\omega)\right)=\Psi(u,+\infty,\omega).
\end{equation*}
Combining the above estimates, we complete the proof for~\eqref{est:pointpsi:Cdeltaregion} in this region.
\end{proof}

 \subsection{Precise decay for zeroth-order mode}\label{SS:0order}

In this section, we establish the precise decay for the zeroth-order mode of $(\phi,\psi,\bar{\psi})$ in the full space-time region. 

\smallskip
In view of ~\eqref{equ:PhiTPsi1}, for the zeroth-mode of $(\phi, \bar{\psi})$, 
we check that 
\begin{equation}\label{equ:PhiTPsi}
\left\{
\begin{aligned}
    UV\Phi_{\ell=0}&=\left[r^{-1}\left(\pt \Psi\right)^{2}\right]_{\ell=0},\\
    UV\bar{\Psi}_{\ell=0}&=\left[r^{-2}\Phi\left(\pt \Psi\right)^{2}\right]_{\ell=0}.  \end{aligned}
\right.
\end{equation} 

\subsubsection{Precise decay of $\phi_{\ell=0}$}
We now derive the precise decay for $\phi_{\ell=0}$ in the full region. We start with the precise decay of $V\Phi_{\ell=0}$ in the space-time region $\mathcal{C}_{{\rm{ext}},2\delta}$.

\begin{lemma}
In the space-time region $\mathcal{C}_{{\rm{ext}},2\delta}$, we have 
\begin{equation}\label{est:VPhi}
\left|V\Phi_{\ell=0}-2\mathfrak{c}_{1}v^{-1}\right|\lesssim \ep v^{-1}u^{-\delta}.
\end{equation}
\end{lemma}
\begin{proof}
Multiplying both sides of the first line of~\eqref{equ:PhiTPsi} by $v$, we have 
\begin{equation*}
U\left(vV\Phi_{\ell=0}\right)=\left[\frac{v}{r}\left(\pt \Psi\right)^{2}\right]_{\ell=0}.
\end{equation*}
Integrating this equation from $(u_{0}(v),v)$ where $u_{0}(v)$ is defined such that 
\begin{equation*}
    (u_{0}(v),v,\omega)\in \mathcal{H}_{1}\Longrightarrow u_{0}(v)=v^{-1},
\end{equation*}
we deduce that
\begin{equation*}
\begin{aligned}
&\left(vV\Phi_{\ell=0}\right)(u,v)-\left(vV\Phi_{\ell=0}\right)(v^{-1},v)\\
&=2\mathfrak{c}_{1}-\frac{1}{4}\int_{v}^{+\infty}\int_{v^{-1}}^{u}V\left[\frac{v}{r}\left(\pt \Psi\right)^{2}\right]_{\ell=0}\left(\sigma,\nu\right)\d \sigma \d \nu\\
&-\int_{u}^{+\infty}\left[\left(\pt \Psi\right)^{2}\right]_{\ell=0}(\sigma,+\infty)\d \sigma-\int_{0}^{v^{-1}}\left[\left(\pt \Psi\right)^{2}\right]_{\ell=0}(\sigma,+\infty)\d \sigma.
\end{aligned}
\end{equation*}
Here, we pushed the integral along constant $v$ to null infinity using the Fundamental Theorem of Calculus.

\smallskip
First, by the smallness condition~\eqref{est:smallness1} of the initial data $(\phi_{0}, \phi_1, \psi_{0}, \psi_1)$, we have
\begin{equation*}
  \left|  \left(vV\Phi_{\ell=0}\right)(v^{-1},v)\right|\lesssim
  \ep v^{-\delta}.
\end{equation*}
Second, using again the almost sharp decay estimate~\eqref{est:almost1}, we deduce
\begin{align*}
\left|\int_{0}^{v^{-1}}\left[\left(\pt \Psi\right)^{2}\right]_{\ell=0}(\sigma,+\infty)\d \sigma   \right|&\lesssim \|\pt \Psi\|_{L^{\infty}}^{2}v^{-1}
\lesssim \ep^2 v^{-1},\\
\left|\int_{u}^{+\infty}\left[\left(\pt \Psi\right)^{2}\right]_{\ell=0}(\sigma,+\infty)\d \sigma\right|
&\lesssim \ep^{2}\int_{u}^{+\infty}(\sigma^{-4}\log^{2}\sigma)\d \sigma\lesssim \ep^2 u^{-3}\log^{2}u.
\end{align*}
Next, by an elementary computation, we find
\begin{equation*}
V\left(\frac{v}{r}\left(\pt \Psi\right)^{2}\right)
=-\frac{u}{r^{2}}\left(\pt \Psi\right)^{2}
+\frac{v}{r}\left(\pt \Psi\right)\left(UV\Psi\right)
+\frac{v}{r}\left(\pt \Psi\right)\left(VV\Psi\right).
\end{equation*}
It follows from~\eqref{est:Omega},~\eqref{equ:PhiPsi:waveUV} and the almost sharp decay estimates~\eqref{est:almost1} that
\begin{equation*}
\begin{aligned}
\left|\left[V\left(\frac{v}{r}\left(\pt \Psi\right)^{2}\right)\right]_{\ell=0}(u,v)\right|
&\lesssim \left|\pt \Psi\right|\left(\frac{u}{r^{2}}\left|\pt \Psi\right|+\frac{v}{r}\left|VV\Psi\right|\right)\\
&+\left|\pt \Psi\right|\left(v\left|Q_{0}(\phi,\phi)\right|+\frac{v}{r^{3}}\left|\Deltas \Psi\right|\right)\lesssim \ep^2 v^{-2+2\delta}u^{-2}\log^{2}u.
\end{aligned}
\end{equation*}
Integrating the above estimate over $ [v^{-1},u]\times
[v,+\infty)$, we obtain
\begin{equation*}
\left|\int_{v}^{+\infty}\int_{v^{-1}}^{u}V\left[\frac{v}{r}\left(\pt \Psi\right)^{2}\right]_{\ell=0}\left(\sigma,\nu\right)\d \sigma \d \nu\right|\lesssim \ep v^{-1+2\delta}.
\end{equation*}
Combining the above estimates, we complete the proof of~\eqref{est:VPhi} in this region.
\end{proof}

We now use the above precise decay estimate for $V\Phi_{\ell=0}$ to derive the precise decay for $\phi_{\ell=0}$ globally in the full space-time region.

\begin{proposition}
[Precise decay for $\phi_{\ell=0}$]\label{Prop:pointphi0}
It holds
\begin{equation}\label{est:pointphi0}
\left|\phi_{\ell=0}(u,v)-\mathfrak{c}_{1}\phi_{L}(u,v)\right|\lesssim \ep u^{-\frac{\delta}{2}}{\phi_{L}}(u,v).
\end{equation}
\end{proposition}

\begin{proof}
We split the proof for the desired estimate \eqref{est:pointphi0} into two parts concerning the space-time regions $\mathcal{C}_{{\rm{ext}},\delta}$ and $\mathcal{C}_{{\rm{int}},\delta}$, respectively. Also, it suffices to consider the case in which $u$ is sufficiently large.

\smallskip
\textbf{Step 1.} Estimate in $\mathcal{C}_{{\rm{ext}},\delta}$.
Integrating ~\eqref{est:VPhi} from $(u,\gamma(u))$ where $\gamma(u)$ is defined such that 
\begin{equation*}
    (u,\gamma(u))\in \big\{r=\frac{1}{2}u^{1-2\delta}\big\}\Longrightarrow \gamma(u)=u\left(1+u^{-2\delta}\right),
\end{equation*}
we deduce that 
\begin{equation}
\lab{eq:rphiell=0:pre}
\begin{aligned}
(r\phi_{\ell=0})(u,v)-(r\phi_{\ell=0})(u,\gamma(u))&=({\mathfrak{c}}_{1}+O(\ep u^{-\delta}))(\ln v-\ln u)\\
&+({\mathfrak{c}}_{1}+O(\ep u^{-\delta}))(\ln u-\ln \gamma(u)).
\end{aligned}
\end{equation}
Note that, using again the almost sharp decay estimates~\eqref{est:almost1}, we have
\begin{equation*}
\begin{aligned}
\left|\frac{\ln u-\ln \gamma(u)}{\ln v-\ln u}\right|&\lesssim  \frac{\ep u^{-2\delta}}{\ln v-\ln u}\lesssim \ep u^{-\delta},\\
\left|\frac{(r\phi_{\ell=0})(u,\gamma(u))}{\ln v-\ln u}\right|
&\lesssim \frac{\ep u^{-2\delta}\ln u}{\ln v-\ln u}\lesssim \ep u^{-\delta}\ln u.
\end{aligned}
\end{equation*}
Combining the above estimates with \eqref{eq:rphiell=0:pre}, and then using~\eqref{eq:roughboundofthepreciseconstants}, we complete the proof of~\eqref{est:pointphi0} in this region.

\smallskip
\textbf{Step 2.} Estimate in $\mathcal{C}_{{\rm{int}},\delta}$.
By the Fundamental Theorem of Calculus, for any $(u,v)\in \mathcal{C}_{{\rm{int}},\delta}$, we compute
\begin{equation*}
\phi_{\ell=0}(u,v)=\phi_{\ell=0}(u,\gamma_{1}(u))+\frac{1}{2}\int_{\gamma_{1}(u)}^{v}V\phi_{\ell=0}(u,\nu)\d \nu.
\end{equation*}
Here, $(u,\gamma_{1}(u))$ is made such that it lies in the curve $r=\frac{1}{2}u^{1-\delta}$, that is,
\begin{equation*}
\gamma_{1}(u)=u+u^{1-\delta}\Longrightarrow \gamma_{1}(u)=u\left(1+u^{-\delta}\right).
\end{equation*}
Using~\eqref{eq:roughboundofthepreciseconstants} and the proven estimate in Step 1, we check that
\begin{equation*}
\begin{aligned}
\phi_{\ell=0}(u,\gamma_{1}(u))
&=\left(2{\mathfrak{c}}_{1}+O(\ep u^{-\frac{\delta}{2}})\right)
\left(\gamma_{1}(u)-u\right)^{-1}\left(\ln \gamma_{1}(u)-\ln u\right)\\
&=2{\mathfrak{c}}_{1}
u^{-1}+O(\ep u^{-1-\frac{\delta}{2}}+|\mathfrak{c}_1|u^{-1-\delta})
=2{\mathfrak{c}}_{1}
u^{-1}+O(\ep u^{-1-\frac{\delta}{2}}).
\end{aligned}
\end{equation*}
Then, the almost sharp decay estimates~\eqref{est:almost1} and $(u,v)\in \mathcal{C}_{\rm{int},\delta}$ imply that
\begin{equation*}
\left|V\phi_{\ell=0}(u,v)\right|
\lesssim \ep v^{-2}\ln v\Longrightarrow
\left|\int_{\gamma_{1}(u)}^{v}V\phi_{\ell=0}(u,\nu)\d \nu\right|
\lesssim \ep u^{-1-\delta}
\ln u.
\end{equation*}
On the other hand, for any $(u,v)\in \mathcal{C}_{{\rm{int}},\delta}$ with $u$ large enough, we have 
\begin{equation*}
\phi_{L}=r^{-1}\ln \left(1+\frac{2r}{u}\right)=2u^{-1}
+O(u^{-1-\delta}).
\end{equation*}
Combining the above estimates, we thus complete the proof of 
~\eqref{est:pointphi0} in 
$\mathcal{C}_{{\rm{int}},\delta}$.
\end{proof}


\subsubsection{Precise decay of $\bar{\psi}_{\ell=0}$}


In this section, we derive the precise decay for $\bar{\psi}_{\ell=0}$. 
We start with the precise decay for $V\bar{\Psi}_{\ell=0}$ in the space-time region $\mathcal{C}_{\rm{ext},2\delta}$. 

\begin{lemma}\label{le:VPsi}
In the space-time region $\mathcal{C}_{\rm{ext},2\delta}$, we have 
\begin{equation}\label{est:VPsi}
\left|V\bar{\Psi}_{\ell=0}-
2\mathfrak{c}_{2}v^{-2}\ln v\right|\lesssim 
\frac{\ep \ln v}{v^{2}\ln u}.
\end{equation}
\end{lemma}

\begin{proof}
Multiplying both sides of the first line of~\eqref{equ:PhiTPsi} by $(v^{2}/\ln v)$, we have 
\begin{equation*}
U\left(\frac{v^{2}}{\ln v}V\bar{\Psi}_{\ell=0}\right)=\left[\left(\frac{v}{r}\right)^{2}\left(\frac{\Phi}{\ln v}\right)\left(\pt \Psi\right)^{2}\right]_{\ell=0}.
\end{equation*}
Integrating this equation from $(u_{0}(v),v)$ where $u_{0}(v)$ is defined such that 
\begin{equation*}
    (u_{0}(v),v,\omega)\in \mathcal{H}_{1}\Longrightarrow u_{0}(v)=v^{-1},
\end{equation*}
we deduce that
\begin{equation*}
\begin{aligned}
&\left(\frac{v^{2}}{\ln v}V\bar{\Psi}_{\ell=0}\right)(u,v)-\left(\frac{v^{2}}{\ln v}V\bar{\Psi}_{\ell=0}\right)(v^{-1},v)\\
&=2\mathfrak{c}_{2}-\frac{1}{4}\int_{v}^{+\infty}\int_{v^{-1}}^{u}V\left[\left(\frac{v}{r}\right)^{2}\left(\frac{\Phi}{\ln v}\right)\left(\pt \Psi\right)^{2}\right]_{\ell=0}\left(\sigma,\nu\right)\d \sigma \d \nu\\
&-2\int_{u}^{+\infty}\left[\left(\frac{\Phi}{\ln v}\right)\left(\pt \Psi\right)^{2}\right]_{\ell=0}(\sigma,+\infty)\d \sigma-
2\int_{0}^{v^{-1}}\left[\left(\frac{\Phi}{\ln v}\right)\left(\pt \Psi\right)^{2}\right]_{\ell=0}(\sigma,+\infty)\d \sigma.
\end{aligned}
\end{equation*}
Here, we pushed the integral along constant $v$ to null infinity using the Fundamental Theorem of Calculus.

\smallskip
First, by the smallness condition~\eqref{est:smallness1} of initial data $(\phi_{0},\psi_{0})$, we have
\begin{equation*}
   \left| \left(\frac{v^{2}}{\ln v}V\bar{\Psi}_{\ell=0}\right)(v^{-1},v)\right|\lesssim \ep v^{-\delta}.
\end{equation*}
Second, using again the almost sharp decay estimate~\eqref{est:almost1}, we deduce
\begin{equation*}
\begin{aligned}
\left|\int_{0}^{v^{-1}}\left[\left(\frac{\Phi}{\ln v}\right)\left(\pt \Psi\right)^{2}\right]_{\ell=0}(\sigma,+\infty)\d \sigma   \right|
&\lesssim \left\|\left(\frac{\Phi}{\ln v}\right)\left(\pt \Psi\right)^{2}\right\|_{L^{\infty}}v^{-1}
\lesssim \ep^3 v^{-1},\\
\left|\int_{u}^{+\infty}\left[\left(\frac{\Phi}{\ln v}\right)\left(\pt \Psi\right)^{2}\right]_{\ell=0}(\sigma,+\infty)\d \sigma\right|
&\lesssim \ep^{3}\int_{u}^{+\infty}\left(\sigma^{-4}\log^{2}\sigma\right)\d \sigma\lesssim \ep^3 u^{-3}\log^{2}u.
\end{aligned}
\end{equation*}
Further, we have from ~\eqref{equ:PhiPsi:waveUV} that
\begin{equation*}
\begin{aligned}
V\left[\left(\frac{v}{r}\right)^{2}\left(\frac{\Phi}{\ln v}\right)\left(\pt \Psi\right)^{2}\right]
&=-2\frac{uv}{r^{3}}\left(\frac{\Phi}{\ln v}\right)\left(\pt \Psi\right)^{2}+\left(\frac{v}{r}\right)^{2}\left(V\left(\frac{\Phi}{\ln v}\right)\right)\left(\pt \Psi\right)^{2}\\
&+\left(\frac{v}{r}\right)^{2}\left(\frac{\Phi}{\ln v}\right)\left(\pt \Psi\right)
\left(r^{-2}\Deltas \Psi+rQ_{0}(\phi,\phi)+VV\Psi\right).
\end{aligned}
\end{equation*}
Hence, using the almost sharp decay estimates~\eqref{est:almost1} as well as 
Lemma~\ref{le:derivativePhiPsi}, we find
\begin{equation*}
\begin{aligned}
&\left|V\left[\left(\frac{v}{r}\right)^{2}\left(\frac{\Phi}{\ln v}\right)\left(\pt \Psi\right)^{2}\right](u,v)\right|\lesssim 
\ep^{3}\left(
v^{-1}u^{-4}(\ln^{-2}v)\ln^{6}u+v^{-2}u^{-3}\ln^{6}v
\right),
\end{aligned}
\end{equation*}
which yields, after integrating the above estimate over $[v^{-1},u]\times [v,+\infty)$, 
\begin{equation*}
\left|\int_{v}^{+\infty}\int_{v^{-1}}^{u}\left(V\left[\left(\frac{v}{r}\right)^{2}\left(\frac{\Phi}{\ln v}\right)\left(\pt \Psi\right)^{2}\right]\right)_{\ell=0}\left(\sigma,\nu\right)\d \sigma \d \nu\right|\lesssim \ep^3 (\ln v)^{-1}.
\end{equation*}
Combining the above estimates, we complete the proof of~\eqref{est:VPsi} in this region.
\end{proof}

We now use the above precise decay estimate for $V\bar{\Psi}_{\ell=0}$ to derive the precise decay for $\bar{\psi}_{\ell=0}$ globally in the full space-time region.

\begin{proposition}
[Precise decay for $\bar{\psi}_{\ell=0}$] \label{Prop:pointpsi0}
It holds in the whole space-time region that
\begin{equation}\label{est:pointpsi0}
\left|\bar{\psi}_{\ell=0}(u,v)-\mathfrak{c}_{2}\psi_{L}(u,v)\right|\lesssim  \ep(\ln u)^{-1}\psi_{L}(u,v).
\end{equation}
\end{proposition}

\begin{proof}
We split the proof into two parts, proving the desired estimates in the space-time regions $\mathcal{C}_{{\rm{ext}},\delta}$ and $\mathcal{C}_{{\rm{int}},\delta}$, respectively. Also, it suffices to consider the case in which $u$ is large enough.

\smallskip
\textbf{Step 1.} Estimate in $\mathcal{C}_{{\rm{ext}},\delta}$.
Integrating ~\eqref{est:VPsi} from $(u,\gamma(u))$ where $\gamma(u)$ is defined such that 
\begin{equation*}
    (u,\gamma(u))\in \big\{r=\frac{1}{2}u^{1-2\delta}\big\}\Longrightarrow \gamma(u)=u\left(1+u^{-2\delta}\right),
\end{equation*}
we deduce
\begin{equation}
\lab{rtildepsiprincipalterminCextde:step1}
\begin{aligned}
&(r\bar{\psi}_{\ell=0})(u,v)-(r\bar{\psi}_{\ell=0})(u,\gamma(u))\\
&=\left(\mathfrak{c}_{2}+O((\ep+|\mathfrak{c}_2|)(\ln u)^{-1})\right)\left(\frac{\ln u}{u}-\frac{\ln v}{v}\right)\\
&+\left(\mathfrak{c}_{2}+O((\ep+|\mathfrak{c}_2|)(\ln u)^{-1})\right)\left(\frac{\ln \gamma(u)}{\gamma(u)}-\frac{\ln u}{u}\right).
\end{aligned}
\end{equation}
Here, we have used the fact that 
\begin{equation*}
    \int_{\gamma(u)}^{v}\sigma^{-2}\d \sigma\lesssim \left(\ln u\right)^{-1}\int_{\gamma(u)}^{v}\sigma^{-2}\ln \sigma \d \sigma.
\end{equation*}
Then, the almost sharp decay estimates~\eqref{est:almost1} imply
\begin{equation}
\lab{esti:rtildepsianderrorterm}
    \left|r\bar{\psi}_{\ell=0}(u,\gamma(u))\right|\lesssim \ep u^{-1-2\delta}\ln^{2}u 
    \ \  \mbox{and}\ \  \left|\frac{\ln \gamma(u)}{\gamma(u)}-\frac{\ln u}{u}\right|\lesssim u^{-1-2\delta}\ln u.
\end{equation}
On the other hand, in view that 
\begin{equation*}
v-u\ge u^{1-\delta}\Longrightarrow 
v\ge u+u^{1-\delta}, \quad \mbox{for any} \ \ (u,v)\in\mathcal{C}_{{\rm{ext}},\delta},
\end{equation*}
and that $(\ln x/x)$ is a strictly decreasing function on $(1,+\infty)$, we obtain
\begin{equation*}
\begin{aligned}
\frac{\ln u}{u}-\frac{\ln v}{v}
&\ge \frac{\ln u}{u}-\frac{\ln (u+u^{1-\delta})}{u+u^{1-\delta}}\\
&\ge \left(\frac{\ln u}{u}-\frac{\ln u}{u+u^{1-\delta}}\right)
-\frac{\ln \left(1+u^{-\delta}\right)}{u+u^{1-\delta}}
\\
&\gtrsim u^{-1-\delta}\ln u-u^{-1-\delta}\gtrsim u^{-1-\delta}\ln u.
\end{aligned}
\end{equation*}
Combining the above estimate with \eqref{rtildepsiprincipalterminCextde:step1} and \eqref{esti:rtildepsianderrorterm},
 we complete the proof for \eqref{est:pointpsi0} in this space-time region.

\smallskip
\textbf{Step 2.} Estimate in $\mathcal{C}_{{\rm{int}},\delta}$.
By the Fundamental Theorem of Calculus, for any $(u,v)\in \mathcal{C}_{{\rm{int}},\delta}$, we compute
\begin{equation*}
\bar{\psi}_{\ell=0}(u,v)=\bar{\psi}_{\ell=0}(u,\gamma_{1}(u))+\frac{1}{2}\int_{\gamma_{1}(u)}^{v}
V\bar{\psi}_{\ell=0}(u,\nu)\d \nu.
\end{equation*}
Here, $(u,\gamma_{1}(u))$ is made such that it lies in the curve $r=\frac{1}{2}u^{1-\delta}$, that is,
\begin{equation*}
\gamma_{1}(u)=u+u^{1-\delta}\Longrightarrow \gamma_{1}(u)=u\left(1+u^{-\delta}\right).
\end{equation*}
Using~\eqref{eq:roughboundofthepreciseconstants} and the proven estimate in Step 1, we check that
\begin{equation*}
\begin{aligned}
    \bar{\psi}_{\ell=0}(u,\gamma_{1}(u))
    &=\left(2\mathfrak{c}_{2}+O((\ep+|\mathfrak{c}_{2}|)(\ln u)^{-1})\right)
    \left(u^{-2}\ln u+O(u^{-2})\right)\\
   & =2\mathfrak{c}_{2}u^{-2}\ln u+O\left((\ep+|\mathfrak{c}_{2}|)u^{-2}\right)=2\mathfrak{c}_{2}u^{-2}\ln u+O(\ep u^{-2}).
    \end{aligned}
\end{equation*}
Then, the almost sharp decay estimates~\eqref{est:almost1} and $(u,v)\in \mathcal{C}_{\rm{int},\delta}$ imply that
\begin{equation*}
    \left|V\bar{\psi}_{\ell=0}(u,v)\right|\lesssim \ep^2 v^{-2}u^{-1}\ln^{2}u\Longrightarrow
    \left|\int_{\gamma_{1}(u)}^{v}V\bar{\psi}_{\ell=0}(u,\nu)\d \nu\right|\lesssim \ep^{2} u^{-2-\delta}\ln^{2}u.
\end{equation*}
On the other hand, for any $(u,v)\in \mathcal{C}_{\rm{int},\delta}$ with $u$ large enough, we have 
\begin{equation*}
    \psi_{L}
    =r^{-1}\left(\frac{1}{u}-\frac{1}{v}\right)\ln u
    +r^{-1}v^{-1}\ln\left(1+\frac{2r}{u}\right)\\
 =
 2u^{-2}\ln u+O(u^{-2}).
\end{equation*}
Combining the above estimates, we thus complete the proof of 
~\eqref{est:pointphi0} in $\mathcal{C}_{{\rm{int}},\delta}$.
\end{proof}


 \subsection{Estimate in Region I} 
 \lab{subsect:precise:regionI}
 

In this space-time region, we expect that the leading order part of the asymptotic behavior for the solution $(\phi,\psi,\bar\psi)$ can be determined from the behavior for the zeroth-order mode $(\phi_{\ell=0},\bar\psi_{\ell=0})$. We show in the following lemma an upper bound of decay in this space-time region for $(\phi_{\ge 1},\psi_{\ge 1})$,  which turns out to decay faster than $(\phi_{\ell=0},\bar{\psi}_{\ell=0})$.

\begin{lemma}\label{le:region1}
    In Region $\mathrm{I}$, we have 
    \begin{subequations}
    \label{eq:esti:geq1modes:regionI}
    \begin{align}
    \label{eq:esti:geq1modes:regionI:phi}
        \|\phi_{\ge 1}(t,r,\omega)\|_{L^{\infty}(\mathbb{S}^{2})}&\lesssim \epsilon v^{-1-\frac{\delta}{2}},\\
        \label{eq:esti:geq1modes:regionI:psi}
    \|\psi_{\ge 1}(t,r,\omega)\|_{L^{\infty}(\mathbb{S}^{2})}&\lesssim \epsilon v^{-2-\frac{\delta}{2}}.
    \end{align}
    \end{subequations}
\end{lemma}
\begin{proof}
Note that, in this space-time region, we have $u\lesssim v\lesssim u$ and $r\lesssim v^{1-\delta}$, which will be constantly used throughout this proof. In view of the wave system~\eqref{equ:main} and the alternative form \eqref{eq:alternativeformsofwaveeq} of wave equation, the system of equations of $(\phi_{\ge 1},\psi_{\ge 1})$ can be written as 
   \begin{equation*}
    \left\{
    \begin{aligned}
        \Deltas\phi_{\ge 1}&=r^{2}UV\phi_{\ge 1}-2r\partial_{r}\phi_{\ge 1}-r^{2}\big[\left(\pt \psi\right)^{2}\big]_{\ge 1},\\
        \Deltas\psi_{\ge 1}&=r^{2}UV\psi_{\ge 1}-2r\partial_{r}\psi_{\ge 1}-r^{2}\big[Q_{0}(\phi,\phi)\big]_{\ge 1}.
    \end{aligned}\right.
 \end{equation*}
 Based on the above identity and the almost sharp decay estimates \eqref{est:almost1}, and applying a standard elliptic estimate over $\mathbb{S}^2$, we deduce
 \begin{equation*}
 \begin{aligned}
      \|\phi_{\ge 1}(t,r,\omega)\|_{H^{2}(\mathbb{S}^{2})}
      &\lesssim \|\Deltas \phi_{\ge 1}(t,r,\omega)\|_{L^{2}(\mathbb{S}^{2})}\\
      &\lesssim \|r^{2}UV\phi\|_{L^{2}(\mathbb{S}^{2})}
      +\|r\partial_{r}\phi\|_{L^{2}(\mathbb{S}^{2})}
      +\|r^{2}\left(\pt \psi\right)^{2}\|_{L^{2}(\mathbb{S}^{2})}\\
      &
      \lesssim \epsilon 
      (v^{-1-2\delta}\ln^{2}v+v^{-1-\delta}\ln v+v^{-4}\ln^{4}v)
      \lesssim \epsilon v^{-1-\frac{\delta}{2}}.
      \end{aligned}
 \end{equation*}
Applying a standard Sobolev embedding theorem on sphere $\mathbb{S}^{2}$, we complete the proof for the desired estimate of $\phi_{\geq 1}$. 

\smallskip
Next, using a similar argument as above, we deduce that 
  \begin{equation*}
 \begin{aligned}
      \|\psi_{\ge 1}(t,r,\omega)\|_{H^{2}(\mathbb{S}^{2})}
      &\lesssim \|\Deltas \psi_{\ge 1}(t,r,\omega)\|_{L^{2}(\mathbb{S}^{2})}\\
      &\lesssim \|r^{2}UV\psi\|_{L^{2}(\mathbb{S}^{2})}
      +\|r\partial_{r}\psi\|_{L^{2}(\mathbb{S}^{2})}
      +\|r^{2}Q_{0}(\phi,\phi)\|_{L^{2}(\mathbb{S}^{2})}\\
      &
      \lesssim \epsilon (v^{-2-\delta}\ln^{2}v+v^{-2-2\delta}\ln^{2}v)\lesssim 
      \epsilon v^{-2-\frac{\delta}{2}}.
      \end{aligned}
 \end{equation*}
Applying a standard Sobolev embedding theorem on sphere $\mathbb{S}^{2}$, we complete the proof for the desired estimate of $\psi_{\geq 1}$. 
\end{proof}

Next, we make use of the above estimates for $(\phi_{\geq 1}, \psi_{\geq 1})$, together with the precise decay estimates for $(\phi_{\ell=0}, \bar{\psi}_{\ell=0})$ proven in Section~\ref{SS:0order}, to compute the precise decay for the solution $(\phi,\bar\psi)$ in Region $\textrm{I}$.

\begin{proposition}[Precise decay for $(\phi,\bar\psi)$ in Region $\textrm{I}$]
\lab{prop:precise:phi:RegionI}
In Region $\mathrm{I}$, we have
\begin{subequations}
\lab{esti:precise:phipsi:RegionI}
\begin{align}\label{esti:precise:phi:RegionI}
\left|\phi(u,v)-\mathfrak{c}_{1}\phi_{L}(u,v)\right|\lesssim& (\ep+|\mathfrak{c}_1|)u^{-\frac{\delta}{2}}\phi_{L}(u,v),\\
\label{esti:precise:psi:RegionI}
\left|\bar\psi(u,v)-\mathfrak{c}_{2}\psi_{L}(u,v)\right|\lesssim& (\ep+|\mathfrak{c}_2|)(\ln u)^{-1}\psi_{L}(u,v),
\end{align}
\end{subequations}
or in an alternative form,
\begin{subequations}
\lab{esti:precise:phipsi:RegionI:other}
\begin{align}\label{esti:precise:phi:RegionI:other}
\left|\phi(u,v)-2\mathfrak{c}_1v^{-1}\right|\lesssim& (\ep+|\mathfrak{c}_1|)v^{-1-\frac{\delta}{2}},\\
\label{esti:precise:psi:RegionI:other}
\left|\bar\psi(u,v)-2\mathfrak{c}_{2}v^{-2}\ln v\right|\lesssim& (\ep+|\mathfrak{c}_2|)v^{-2}.
\end{align}
\end{subequations}
\end{proposition}

\begin{proof}
The estimate \eqref{esti:precise:phi:RegionI} for $\phi$ follows from combining the estimate \eqref{est:pointphi0} for $\phi_{\ell=0}$ with the estimate \eqref{eq:esti:geq1modes:regionI:phi} for $\phi_{\geq 1}$, and the alternative estimate \eqref{esti:precise:phi:RegionI:other} for $\phi$ follows from \eqref{esti:precise:phi:RegionI} together with the following fact
\begin{equation*}
\phi_{L}=-r^{-1}\ln \bigg(1-\frac{2r}{v}\bigg)= 2v^{-1} + O(rv^{-2}) = 2v^{-1} + O(v^{-1-\de}) \quad \textrm{in\,\, Region\,\, I}.
\end{equation*}

  Combining the estimate \eqref{est:pointpsi0} for $\bar{\psi}_{\ell=0}$ with the estimate \eqref{eq:esti:geq1modes:regionI:psi} for $\psi_{\geq 1}$, and in view of the estimate \eqref{esti:precise:phi:RegionI:other} and the following fact
\begin{align*}
\psi_{L}={}&r^{-1} \bigg(\frac{\ln u}{u} - \frac{\ln v}{v}\bigg)=r^{-1}\bigg(\frac{2r\ln u }{u v} + \frac{\ln v - \ln u }{v}
\bigg)
\\
={}&\frac{2\ln u }{uv} +O(v^{-2})=\frac{2\ln v}{v^2} + O(v^{-2}) \quad \textrm{in\,\, Region\,\, I},
\end{align*}
the other estimates \eqref{esti:precise:psi:RegionI} and  \eqref{esti:precise:psi:RegionI:other} then follow.
\end{proof}

\begin{remark}\label{re:region12}
Note that, using the Cauchy-Schwarz inequality, for $u\in (1,+\infty)$ large enough,
 \begin{equation}\label{est:phiLpsiL}
\begin{aligned}
    \phi_{L}^{2}=r^{-2}\left(\int_{u}^{v}\sigma^{-1}\d \sigma\right)^{2}
    \lesssim r^{-1}\int_{u}^{v}\sigma^{-2}\d \sigma\lesssim r^{-1}\int_{u}^v \frac{\ln \sigma }{\sigma^2} \d \sigma\lesssim 
    \frac{\psi_{L}}{\ln u}.
\end{aligned}
\end{equation}
Hence, from Propositions~\ref{prop:pointphi}, \ref{Prop:pointpsi0} and \ref{prop:precise:phi:RegionI} and the almost sharp decay estimate \eqref{est:almost1}, it follows that in the large space-time region $\left\{2r\le u^{1-\delta}\right\}\cup\left\{2r\ge u^{1+\delta}\right\}$, the leading order terms in the precise decay of $(\psi,\psi_{\ell=0})$ are the same as the ones of $(\bar{\psi},\bar{\psi}_{\ell=0})$.
\end{remark}

\section{Genericity of the sharp decay rates}\label{sec:generic}

In this section, we prove that for a generic set of small initial data, the constants $\{\cfrak_{i}\}_{i=1,2}$ are non-zero and the functions on the sphere $\{\cfrak_i\}_{i=3,4}$ do not vanish identically. We start with the definition of the admissible initial data set. 

\begin{definition}[Admissible initial data set $\Sfrak$]
\label{def:admissibledata}
A pair of initial data $(\phiinit,\psiinit)=(\phi_0,\phi_1,\psi_0,\psi_1)$ on $\mathcal{H}_1$ is called  \textit{admissible initial data} if it satisfies the condition \eqref{est:smallness1} for a given $N\geq 6$ with $0<\ep\ll 1$ suitably small such that the global existence and decay estimates statement in Theorem \ref{thm:main1} for the solution $(\phi,\psi)$ to the system \eqref{equ:main} with initial data $(\phiinit,\psiinit)=(\phi_0,\phi_1,\psi_0,\psi_1)$  hold.
The set of all admissible initial data is called the admissible initial data set and denoted as $\Sfrak$.
\end{definition}

We next define a distance function on the admissible initial data set $\Sfrak$.

\begin{definition}[Distance function and its induced topology]
\label{def:distancefunction}
Let $N\geq 6$ be given. Let $(\phiinit,\psiinit)=(\phi_0,\phi_1,\psi_0,\psi_1)$ and $(\phihinit,\psihinit)=(\hat\phi_0,\hat\phi_1,\hat\psi_0,\hat\psi_1)$ be two pairs of admissible initial data in $\Sfrak$, and let $(\phi,\psi)$ and $(\hat\phi,\hat\psi)$ be the global solutions to the system of wave equations \eqref{equ:main} arising from these pairs of initial data respectively. Define a distance function for these two pairs of admissible initial data by
\begin{align}
\lab{eq:distancefunction}
\dist{(\phiinit,\psiinit),(\phihinit,\psihinit)}
:={}&   \|\langle x\rangle^{{N}+2} \partial_x^{\leq N+1} (\phi_0-\hat\phi_0,  \psi_0-\hat\psi_0)\|\notag\\
    &+  \|\langle x\rangle^{{N}+2} \partial_x^{\leq N} (\phi_1-\hat\phi_1, \psi_1-\hat\psi_1)\|.
\end{align}
This distance function trivially induces a natural topology in the admissible initial data set $\Sfrak$. 
\end{definition}

Note that the right-hand side of \eqref{eq:distancefunction} is the norm showing up on the left-hand side of the smallness condition \eqref{est:smallness1}, but for the difference of the two given initial data  $(\phiinit,\psiinit)$ and $(\phihinit,\psihinit)$.

Recall the definition of the constants $(\cfrak_1, \cfrak_2)$ and functions $(\cfrak_3(\omega), \cfrak_4(\omega))$ appearing in the principal term of the late-time asymptotics of $(\phi,\psi)$ in the main theorem \ref{thm:main1}. In view of $2\partial_t = U+V$, and since applying once $V$ derivative gains an extra $v^{-1}$ decay as shown in the decay estimates \eqref{est:almost1}, we have
\begin{subequations}
\label{def:cfrakanddfrakconsts:copy}
\begin{align}
\cfrak_1=&\frac{1}{32\pi }\int_{\mathbb{S}^2}\int_{0}^{+\infty}(U (r\psi))^2 (u,+\infty,\omega)\d u \d\omega, \\
\cfrak_2=&\frac{1}{16\pi}\int_{\mathbb{S}^2}\int_{0}^{+\infty}\bigg( \frac{r\phi}{\ln v} (U (r\psi))^2\bigg) (u,+\infty,\omega)\d u \d\omega, \\
 \cfrak_3(\omega)=&\frac{1}{8}\int_{0}^{+\infty}(U (r\psi))^2 (u,+\infty, \omega) \d u,\\
\cfrak_4(\omega)=&\frac{1}{4}\int_{0}^{+\infty} \bigg(\frac{r\phi}{\ln v} (U(r\psi))^2 \bigg)(u,+\infty,\omega)\d u.
\end{align}
\end{subequations}

We introduce the sets
$
(\Sfrak_{\neq 0, i})_{i=1}^{4},
$
where the sets $\Sfrak_{\neq 0, 1}$ and $\Sfrak_{\neq 0, 2}$ are composed by the admissible initial data such that the constants $\cfrak_i\neq 0$ for $i\in \{1,2\}$, and where the sets $\Sfrak_{\neq 0, 3}$ and $\Sfrak_{\neq 0, 4}$ are composed by the admissible initial data such that there exists an $\omega\in\mathbb{S}^2$ such that $\cfrak_i(\omega)\neq 0$ for $i\in\{3,4\}$. Further, we define the following subset of the admissible initial data set
\begin{equation*}
\Sfrak_{\neq 0}:=\Sfrak_{\neq 0, 1}\cap \Sfrak_{\neq 0, 2}\cap\Sfrak_{\neq 0, 3}\cap \Sfrak_{\neq 0, 4}.
\end{equation*}

Our objective in this section is to show that the subset $\Sfrak_{\neq 0}$ is an open and dense subset of $\Sfrak$ with respect to the induced topology introduced in Definition \ref{def:distancefunction}. In other words, we show that, for generic initial data in the admissible initial data set $\Sfrak$, both of the constants $(\cfrak_1,\cfrak_{2})$ are non-vanishing and neither of the functions $(\cfrak_3(\omega),\cfrak_{4}(\omega))$ are identically vanishing over $\omega\in\mathbb{S}^2$.

\smallskip
Since $(\cfrak_{1},\cfrak_{2})$ are the spherical means of $(\cfrak_{3}(\omega),\cfrak_{4}(\omega))$, respectively, it holds true that  $\Sfrak_{\neq 0, 1}\subset \Sfrak_{\neq 0, 3}$ and $\Sfrak_{\neq 0, 2}\subset \Sfrak_{\neq 0, 4}$, and thus,
\begin{equation}\label{eq:1-11-sept-2025}
\Sfrak_{\neq 0} 
=\Sfrak_{\neq 0, 1}\cap \Sfrak_{\neq 0, 2}\cap\Sfrak_{\neq 0, 3}\cap \Sfrak_{\neq 0, 4}
= \Sfrak_{\neq 0, 1}\cap \Sfrak_{\neq 0, 2}.
\end{equation}
To show $\Sfrak_{\neq 0}$ is open and dense, it suffices to prove that $\Sfrak_{\neq 0, 1}$ and $\Sfrak_{\neq 0, 2}$ are open and dense, because the intersection of two open-dense sets is still open and dense.

Notice that the constants $\{\cfrak_i[\phi,\psi]\}_{i=1,2}$ are continuous functions of the solution $(\phi,\psi)$, and are hence continuously dependent on its initial data $(\phiinit,\psiinit)$ in view of the global well-posedness of the system~\eqref{equ:main}. Therefore, the subsets $\Sfrak_{\neq 0, i}$ for $i\in \{1,2\}$ are open. As a consequence, to establish that $\Sfrak_{\neq 0}$ is open and dense, it remains to show that the sets $\{\Sfrak_{\neq 0, i}\}_{i=1,2}$ are both dense in $\Sfrak$, which is the main task of the following subsections~\ref{subsect:density:genericity} and \ref{subsect:density:genericity:G2neq0}.


\subsection{The subset $\Sfrak_{\neq 0, 1}$ is dense in $\Sfrak$}
\label{subsect:density:genericity}
We start with a technical lemma, asserting a faster decay in $v$ for the term $U\Psi$ in the case of $\mathfrak{c}_1=0$.

\begin{lemma}\label{lem:UPsi-v-decay}
    Assume that $\mathfrak{c}_1=0$. Then we have
    \begin{align*}
        |U\Psi|+|V\Psi|
        \lesssim
        \epsilon v^{-1}u^{-1}\ln^{2}v.
    \end{align*}
\end{lemma}
\begin{proof}
First, from~\eqref{est:almost1}, we directly have 
\begin{equation*}
    |V\Psi|
        \lesssim
        \epsilon v^{-1}u^{-1}\left(\ln u+(u/v)\ln^{2}v\right)
        \lesssim
        \epsilon v^{-1} u^{-1}\ln^{2}v.
\end{equation*}
Second, from the definition of $\mathfrak{c}_{1}$ in~\eqref{def:cfrakanddfrakconsts:copy}, we find 
\begin{equation*}
    \mathfrak{c}_{1}=\frac{1}{32\pi }\int_{\mathbb{S}^2}\int_{0}^{+\infty}(U (r\psi))^2 (u,+\infty,\omega)\d u \d\omega=0\Longrightarrow U\Psi\equiv 0,\ \ \mbox{on}\ \ \mathcal{I}.
\end{equation*}
It follows from the equation  $VU\Psi = r^{-2} \slashed{\Delta} \Psi + r Q_0(\phi, \phi)$ that 
\begin{equation*}
    |U\Psi(u, v)|
    \lesssim
    \int_v^{+\infty} \big|  r^{-2} \slashed{\Delta} \Psi + r Q_0(\phi, \phi)  \big|(u, \nu) \, d\nu.
\end{equation*}
By the almost sharp decay~\eqref{est:almost1}, we find 
\begin{equation*}
    \big|  r^{-2} \slashed{\Delta} \Psi + r Q_0(\phi, \phi)  \big|(u, v)
    \lesssim
    \epsilon v^{-2}u^{-1}\ln^{2}v,
\end{equation*}
which implies
\begin{equation*}
    |U\Psi(u, v)|
    \lesssim
    \epsilon u^{-1} \int_v^{+\infty}  \nu^{-2}\ln^{2}\nu \d \nu
    \lesssim
    \epsilon v^{-1} u^{-1}\ln^{2}v.
\end{equation*}
Combining the above estimates, we complete the proof for Lemma \ref{lem:UPsi-v-decay}.
\end{proof}

We are in a position to show that the subset $\Sfrak_{\neq 0, 1}$ is dense in $\Sfrak$.

\begin{proof}[Proof of denseness of $\Sfrak_{\neq 0, 1}$ in $\Sfrak$]

The proof proceeds by contradiction.
Assume otherwise that there exists a pair of initial data $(\phi_0, \phi_1, \psi_0, \psi_1)\in \Sfrak$ such that:
\begin{enumerate}[label=\arabic*)]
    \item $\mathfrak{c}_1(\phi_0, \phi_1, \psi_0, \psi_1) = 0$;
    \item\label{point2:contradictionofdense}
    There exists $0<\epsilon_1\ll1$ such that $\mathfrak{c}_1(\phi'_0, \phi'_1, \psi'_0, \psi'_1)=0$ for all initial data $(\phi'_0, \phi'_1, \psi'_0, \psi'_1)$  belonging to the deleted $\epsilon_1$-neighborhod of $(\phi_0, \phi_1, \psi_0, \psi_1)$ in $\Sfrak$.
\end{enumerate}

We denote by $(\check{\phi}, \check{\psi})=(\phi - \phi', \psi-\psi')$ the difference of the two solution pairs, in which the pair $(\phi', \psi')$ solves \eqref{equ:main} with initial data $(\phi'_0, \phi'_1, \psi'_0, \psi'_1)$.

\smallskip
\textbf{Step 1.} Bounds for $\check{\psi}$.
We claim that 
\begin{equation}\label{eq:psi-check}
    \sup_{s\in [1, +\infty)} \mathcal{E}(s, \check{\psi})
    \lesssim
    \epsilon \sup_{s\in [1, +\infty)} \mathcal{E}(s, \check{\phi}).
\end{equation}
Indeed, we consider the following equation for the difference $\check{\psi}$ with $\check{\Psi}=r\check{\psi}$,
\begin{align*}
    UV \check{\Psi} - r^{-2} \slashed{\Delta} \check{\Psi}
    =
    rQ_0(\phi, \check{\phi}) + rQ_0(\check{\phi}, \phi').
\end{align*}
By the standard energy estimate in the space-time region $\mathcal{H}_{[s,+\infty)}$ with the boundary $\mathcal{H}_{s}\cup \mathcal{I}$ and then using the fact that $U\Psi=0$ on $\mathcal{I}$,
\begin{equation*}
    \mathcal{E}(s, \check{\psi}) 
    \lesssim
    \int_{s}^{+\infty} \int_{0}^{+\infty}\int_{\mathbb{S}^2} {\tau\over t}\Big| rQ_0(\phi, \check{\phi}) + rQ_0(\check{\phi}, \phi') \Big| \big|\partial_t \check{\Psi}\big| \, \d \omega \d r \d\tau.
\end{equation*}
Note that, from $\tau=\sqrt{t^{2}-r^{2}}$ and the almost sharp decay~\eqref{est:almost1}, we obtain
\begin{equation*}
\begin{aligned}
    &{\tau\over t}\Big| rQ_0(\phi, \check{\phi}) + rQ_0(\check{\phi}, \phi') \Big| \big|\partial_t \check{\Psi}\big|
    \lesssim \tau^{-\frac{5}{4}} \Big|\frac{\tau r}{t} U\check{\psi}\Big|
    \Big(\left|rV\check{\phi}\right|+\Big|\frac{\tau r}{t} U\check{\phi}\Big|\Big),
    \end{aligned}
\end{equation*}
which thus implies 
\begin{align*}
    \mathcal{E}(s, \check{\psi}) 
    &\lesssim
    \epsilon \int_s^{+\infty} \tau^{-{5\over 4}} \mathcal{E}(\tau, \check{\phi})^{1\over 2} \mathcal{E}(\tau, \check{\psi})^{1\over 2} \, \d\tau
    \\
    &\lesssim
    \epsilon \int_s^{+\infty} \tau^{-{5\over 4}} \big(\mathcal{E}(\tau, \check{\phi}) + \mathcal{E}(\tau, \check{\psi}) \big) \, \d\tau\\
    &\lesssim
    \epsilon \sup_{s\in [1, +\infty)}\mathcal{E}(s,\check{\phi}) + \epsilon\sup_{s\in [1, +\infty)}\mathcal{E}(s, \check{\psi}).
\end{align*}
By smallness of $\epsilon$, we complete the proof of the desired estimate~\eqref{eq:psi-check}.

\smallskip
\textbf{Step 2.} Bounds for $\check{\phi}$.
We claim that 
\begin{equation}\label{eq:phi-check}
    \sup_{s\in [1, +\infty)} \mathcal{E}(s, \check{\phi})
    \lesssim
    \mathcal{E}(1, \check{\phi})
    +
    \epsilon \sup_{s\in [1, +\infty)} \mathcal{E}(s, \check{\psi}).
\end{equation}
Note that the energy bound for $\check{\phi}$ generally exhibits a logarithmic divergence. However, under the assumption $\mathfrak{c}_{1} = 0$, a standard energy estimate for the wave equation satisfied by $\check{\phi}$ over $\mathcal{H}_{[1,s]}$ yields a more refined energy bound. More precisely, by an elementary computation, we have 
\begin{align*}
    -\Box \check{\phi}
    =
    \partial_t \check{\psi} (\partial_t \psi + \partial_t \psi'),
\end{align*} 
which directly implies
\begin{align*}
    \mathcal{E}(s, \check{\phi})
    \lesssim{}
    &\mathcal{E}(1, \check{\phi})
    +
    \int_1^s \int_{\mathcal{H}_\tau} {\tau\over t}\big| \partial_t \check{\psi} (\partial_t \psi + \partial_t \psi') \big| \cdot \big| \partial_t \check{\phi}  \big| \, \d x \d\tau.
\end{align*}
Note that, from the almost sharp decay~\eqref{est:almost1} and Lemma~\ref{lem:UPsi-v-decay}, we obtain
\begin{equation*}
    |\partial_t \psi |
    +|\partial_t \psi' |
    \lesssim \min \left(v^{-1}u^{-2}\ln^{2}v,r^{-1}v^{-1}u^{-1}\ln^{2}v\right).
\end{equation*}
It follows that 
\begin{align*}
    \mathcal{E}(s, \check{\phi})
    &\lesssim
    \mathcal{E}(1, \check{\phi})
    +
    \ep\int_1^s \tau^{-{5\over 4}} \big( \mathcal{E}(\tau, \check{\phi}) + \mathcal{E}(\tau, \check{\psi})\big) \, \d\tau\\
    &\lesssim
    \mathcal{E}(1, \check{\phi})
    +
    \epsilon \sup_{s\in [1, +\infty)}\mathcal{E}(s, \check{\phi}) +  \epsilon \sup_{s\in [1, +\infty)}\mathcal{E}(s, \check{\psi}).
\end{align*}
By smallness of $\epsilon$, we complete the proof of the desired estimate \eqref{eq:phi-check}.

\smallskip
\textbf{Step 3.} Conclusion. Combining \eqref{eq:psi-check} and \eqref{eq:phi-check}, we arrive at
\begin{align*}
    \sup_{s\in [1, +\infty)} \mathcal{E}(s, \check{\psi})
    \lesssim
    \epsilon \mathcal{E}(1, \check{\phi})
    +
    \epsilon^2 \sup_{s\in [1, +\infty)} \mathcal{E}(s, \check{\psi})\Longrightarrow \mathcal{E}(1, \check{\psi})
    \lesssim
    \epsilon \mathcal{E}(1, \check{\phi}),
\end{align*}
which contradicts the condition \ref{point2:contradictionofdense} since this estimate can not hold everywhere in the deleted $\epsilon_1$-neighborhod of $(\phi_0, \phi_1, \psi_0, \psi_1)$ in $\Sfrak$. This completes the proof that the subset $\Sfrak_{\neq 0, 1}$ is dense in $\Sfrak$.
\end{proof}

\subsection{The subset $\Sfrak_{\neq 0, 2}$ is dense in $\Sfrak$}
\label{subsect:density:genericity:G2neq0}

Recall from \eqref{est:Firsttech} that, in the space-time region $\mathcal{D}_{\rm{ext},2\delta}$, we have
\begin{equation*}
\left|U\Phi-\ln v \left(\pt \Psi\right)^{2}\right|\lesssim \ep u^{-1}\ln u\Longrightarrow \left(U\bigg(\frac{\Phi}{\ln v}\bigg)\right)(u,+\infty,\omega)= \frac{1}{4}(U \Psi)^2(u,+\infty,\omega).
\end{equation*}
It follows from the Fundamental Theorems of Calculus that
\begin{equation*}
\left(\frac{r\phi}{\ln v}\right){(u,+\infty, \omega)}= \frac{1}{4}\int_{0}^{u}(U \Psi)^2(\sigma,+\infty,\omega)\d\sigma, \quad \mbox{for any} \ \ u>0.
\end{equation*}
Plugging this into the formulas of $\mathfrak{c}_2$ and $\mathfrak{c}_4(\omega)$ in \eqref{def:cfrakanddfrakconsts:copy}, we infer that $\mathfrak{c}_2=0$ is equivalent to  $\mathfrak{c}_1=0$, and  $\mathfrak{c}_4(\omega)=0$ is equivalent to $\mathfrak{c}_3(\omega)=0$. Therefore, the statement that the subset $\Sfrak_{\neq 0, 2}$ is dense in $\Sfrak$ follows from  the above proven statement in Section \ref{subsect:density:genericity} that the subset $\Sfrak_{\neq 0, 1}$ is dense in $\Sfrak$.

\section{Proof of Theorem \ref{thm:higher}}\label{sec:appendix}

Fix $\ell \in \mathbb{Z}_+$, and consider the spacetime region $\mathcal{M}_{\ell} =\{ r\geq u^{1-\delta_\ell} \}$ with $\delta_\ell \leq \min\{ {\delta\over 2},{1\over 4\ell + 4}\}$. 
With a slight abuse of notation, below we write
\begin{align*}
    \Phi = \Phi_{\ell}, \qquad F = \big((\partial_t \Psi)^2\big)_{\ell}.
\end{align*}
The equation of $\Phi$ then reads
\begin{align*}
    r^2 U V \Phi + \ell (\ell+1) \Phi = rF.
\end{align*}
Acting $(r^2 V)^\ell$ on both sides of the equation, we obtain
\begin{align*}
    r^2 U V \Phi^{(\ell)} + 2\ell r V\Phi^{(\ell)} = (r^2 V)^\ell (rF),
\end{align*}
in which
\begin{align*}
    \Phi^{(k)} := (r^2 V)^k \Phi, \quad \forall \,\, k\in\mathbb{N}.
\end{align*}
Rescaling by $r^{-2\ell -2}$ gives
\begin{align*}
    U(r^{-2\ell} V\Phi^{(\ell)})
    =
    r^{-2\ell-2} (r^2 V)^\ell (rF),
\end{align*}
which, by multiplying by $v^{\ell+1}$, leads to
\begin{align*}
    U(v^{\ell+1} r^{-2\ell-2} \Phi^{(\ell+1)})
    =
    v^{\ell+1}r^{-2\ell-2} (r^2 V)^\ell (rF).
\end{align*}
We integrate from $\mathcal{H}_1$ along $U$ direction (i.e., the integral curve of $U$) to get
\begin{align*}
    \big(v^{\ell+1} r^{-2\ell-2} \Phi^{(\ell+1)}\big)(u, v)
    -
    \big(v^{\ell+1} r^{-2\ell-2} \Phi^{(\ell+1)}\big)(v^{-1}, v)
    =
    {1\over 2}\int_{v^{-1}}^u G ,
\end{align*}
with
\begin{align*}
    G
    = {}&v^{\ell+1}r^{-2\ell-2} (r^2 V)^\ell (rF)
    \\
    = {}&v^{\ell+1} r^{-2\ell-2} (r^2 V)^{\ell}r \cdot F
    +
    v^{\ell+1} r^{-2\ell-2} \sum_{i=0}^{\ell-1} \binom{\ell}{i} (r^2 V)^{i} r \cdot (r^2 V)^{\ell-i}F
    \\
    ={}
    &\ell! v^{\ell+1} r^{-\ell-1} F
    +
     v^{\ell+1} r^{-2\ell-2} \sum_{i=0}^{\ell-1} 
\binom{\ell}{i} 
i! r^{i+1} \cdot (r^2 V)^{\ell-i}F.
\end{align*}
By applying the equation of $\Psi$, we have
\begin{align*}
    (r^2 V) F = {}&2r^2 \big(\partial_t \Psi \partial_t V \Psi\big)_{\ell}
    \\
    = {}&r^2 \big(\partial_t \Psi (VV\Psi -r^{-2} \Deltas\Psi + r Q_{0}(\phi, \phi))\big)_{\ell}
    \\
    ={}
    &\mathcal{O}(\ep^2 u^{-2+\delta} (\ln{v})^2),
\end{align*}
and more generally for $i= 0, 1, \cdots, \ell-1$, we have
\begin{align*}
    |(r^2 V)^{\ell-i} F|
    \lesssim_\ell \ep^2
    u^{-2+\delta} (\ln{v})^2 r^{\ell-i-1},
\end{align*}
which then yields
\begin{align*}
    G = \ell! v^{\ell+1} r^{-\ell-1} F
    + v^{\ell+1} r^{-\ell-2} \mathcal{O}_\ell (\ep^2 u^{-2+\delta} (\ln{v})^2).
\end{align*}
Consequently, we have, for $(u,v)\in\{ r\geq u^{1-2\delta_\ell} \}$,
\begin{align*}
    \int_{v^{-1}}^u G
    =2^{\ell+1} \ell! \mathcal{C}_\ell + \mathcal{E}_{rr, \ell},
\end{align*}
in which 
\begin{equation}
\lab{esti:CellandErrell}
\mathcal{C}_\ell = \int_{\mathcal{I}} F,
\qquad
\mathcal{E}_{rr, \ell}
= \mathcal{O}_\ell (\ep^2  v^{-1+ 2(\ell+2)\delta_\ell} (\ln{v})^2).
\end{equation}
Till now, we have derived
\begin{align}
\lab{esti:Phiell+1}
    \Phi^{(\ell+1)}(u, v)
    =
    v^{-(\ell+1)} r^{2\ell+2} (2^{\ell} \ell! \mathcal{C}_\ell + \mathcal{E}_{rr, \ell}).
\end{align}
We introduce the Bondi-Sachs coordinates $(u, R=r^{-1}, \omega)$, and in this coordinate system one has
\begin{align*}
    {\partial}^{BS}_u = \partial_t,
    \qquad
    \partial^{BS}_R = -r^2 V,
\end{align*}
as well as the relation
$$
v= {2+uR \over R}.
$$
Now, we have
\begin{align*}
    \Phi^{(\ell+1)} = (-\partial_R^{BS})^{\ell+1} \Phi.
\end{align*}
Hence, equation \eqref{esti:Phiell+1} can be written as
\begin{align}
\lab{esti:PhipartialRell+1derivatives}
   (-\partial_R^{BS})^{\ell+1} \Phi(u, v)
    =
    (2+uR)^{-(\ell+1)} R^{-\ell-2} (2^{\ell} \ell! \mathcal{C}_\ell + \mathcal{E}_{rr, \ell}).
\end{align}
Integrating this equation from $\Gamma := \{R = u^{-1+2\delta_{\ell}} \}$  along $-\partial_R^{BS}$ direction yields
\begin{align*}
    &(-1)^{\ell+1}\Phi(u, R) + \mathcal{B}_\ell
    \\
    ={}
    &\int_{R}^{u^{-1+2\delta_{\ell}}} \int_{R_\ell}^{u^{-1+2\delta_{\ell}}} \cdots \int_{R_1}^{u^{-1+2\delta_{\ell}}} {2^{\ell} \ell! \mathcal{C}_\ell + \mathcal{E}_{rr, \ell} \over (2+u\rho)^{\ell+1} \rho^{\ell+1}} \, d\rho dR_1 \cdots dR_{\ell}
    \\
    ={}
    &\mathcal{C}_\ell \int_{R}^{u^{-1+2\delta_{\ell}}} {(\rho-R)^\ell \over 2\rho^{\ell+1} (1+{u\rho\over 2})^{\ell+1}}   \, d\rho
    +\int_{R}^{u^{-1+2\delta_{\ell}}} {(\rho-R)^\ell \over 2\rho^{\ell+1} (1+{u\rho\over 2})^{\ell+1}}  {1\over 2^{\ell+1} \ell!}\mathcal{E}_{rr, \ell}  \, d\rho
    \\
    :={}
    &\mathfrak{G}_1 + \mathfrak{G}_2,
\end{align*}
in which $\mathcal{B}_\ell$ represents the contribution from the boundary $\Gamma$ with a bound from \eqref{eq:esti:geq1modes:regionI:phi}
\begin{align*}
    |\mathcal{B}_\ell|
    \lesssim_{\ell}
   \ep u^{- \delta_{\ell}}.
\end{align*}

By a simple change of variable $z=u\rho$, we find
\begin{align*}
    \mathfrak{G}_1
    =
    \mathcal{C}_\ell \int_{uR}^{u^{2\delta_{\ell}}} {(z-uR)^{\ell} \over 2z^{\ell+1} (1+ {z\over 2})^{\ell+1}} \, dz
    =
    \mathfrak{G}_{1a} + \mathfrak{G}_{1b},
\end{align*}
in which
\begin{align*}
    \mathfrak{G}_{1a}
    ={}
    &\frac{1}{2}\mathcal{C}_\ell \mathcal{D}_\ell({u r^{-1}}),
    \qquad
    \mathcal{D}_\ell({u r^{-1}})
    =\int_{uR}^{+\infty} {(z-uR)^{\ell} \over z^{\ell+1} (1+ {z\over 2})^{\ell+1}} \, dz,
    \\
    \mathfrak{G}_{1b}
    =
    &-\frac{1}{2}\mathcal{C}_\ell \int_{u^{2\delta_{\ell}}}^{+\infty} {(z-uR)^{\ell} \over z^{\ell+1} (1+ {z\over 2})^{\ell+1}} \, dz
    =
    \mathcal{O}_\ell (u^{-2(\ell+1)\delta_{\ell}}).
\end{align*}

Thanks to the bound \eqref{esti:CellandErrell} of $\mathcal{E}_{rr, \ell}$, we get
\begin{align*}
    |\mathfrak{G}_2|
    \lesssim_{\ell}
     \ep^2 u^{-\delta_{\ell}},
\end{align*}
which is a lower-order term.

Therefore, we infer in $\mathcal{M}_\ell$ that
\begin{align}
    \phi_\ell (u, r)
    =
     \frac{(-1)^{\ell+1}}{2r} \mathcal{C}_\ell \mathcal{D}_\ell({u r^{-1}})
    +
    \mathcal{O}_{\ell}(\ep r^{-1}u^{-\delta_\ell}).
\end{align}

\section{Proof of Theorem \ref{thm:blow-up}}\label{sec:growth}

We are in a position to complete the proof for Theorem~\ref{thm:blow-up}. The proof is based on the almost sharp decay~\eqref{est:almost1} and some standard energy estimates. We start with the following growth and boundedness estimates for  $\phi$ and $(\psi,\partial \psi)$.
\begin{lemma}
\label{lem:growthofL2ofphi}
   Under the assumptions of Theorem~\ref{thm:blow-up}, the following estimates hold for sufficiently large $t$.
   \begin{enumerate}
       \item \emph{(Growth estimates for $\phi$).} We have, 
    \begin{equation*}
        \mathfrak{c}_{1} t^{\frac{1}{2}}\lesssim \|\phi(t)\|\lesssim \epsilon t^{\frac{1}{2}}\ln t.
    \end{equation*}

    \item \emph{(Bounds for $L^{2}$ norm of $(\psi,\partial\psi)$).} We have 
    \begin{equation*}
        \|\psi\|_{L^{2}}+\|\partial \psi\|_{L^{2}}\lesssim \epsilon.
    \end{equation*}
    \end{enumerate}
\end{lemma}

\begin{proof}
    Proof of (i). Recall that, from the almost sharp decay~\eqref{est:almost1}, we find 
    \begin{equation*}
        |\phi(t,r,\omega)|\lesssim \epsilon (t+r)^{-1}\ln (t+r).
    \end{equation*}
    Based on the above estimate and the assumption that $(\phi_{0},\phi_{1},\psi_{0},\psi_{1})$ has compact support, we have
    \begin{equation*}
    \begin{aligned}
        \|\phi(t)\|^{2}
        &\lesssim \int_{\mathbb{S}^{2}}\int_{\R}|\phi(t,r,\omega)|^{2}r^{2}\d r\d \omega\\
        &\lesssim \epsilon^{2}\int_{0}^{t} \left(\frac{r}{t+r}\right)^{2}\ln^{2}(t+r)\d r\lesssim 
        \epsilon^{2}t \ln^{2}t.
        \end{aligned}
    \end{equation*}
    Then, from the fact that $u=t-r$ and $v=t+r$,
    \begin{equation*}
        r\in \left(\frac{t}{8},\frac{t}{4}\right)\Longrightarrow u\sim v\sim t\sim r.
    \end{equation*}
    It follows from the precise decay of $\phi_{\ell=0}$ in \eqref{est:pointphi0} and $\|\phi\| \geq \|\phi_{\ell=0}\|$ that 
    \begin{equation*}
        \|\phi(t)\|^{2}\gtrsim \int_{\frac{t}{8}}^{\frac{t}{4}}|\phi_{\ell=0}(t,r)|^{2}r^{2}\d r\gtrsim (\cfrak_1)^2t.
    \end{equation*}
    Combining the above estimates, we complete the proof for the estimates in (i).

    \smallskip
    Proof of (ii). Recall also that, from the almost sharp decay~\eqref{est:almost1},
    \begin{equation*}
        |\psi(t,r,\omega)|+|\partial \psi(t,r,\omega)|\lesssim \epsilon \langle t+r\rangle^{-1}\langle t-r\rangle^{-1}\ln^{2}\langle t-r\rangle.
    \end{equation*}
    Based on the above estimate and the assumption that  $(\phi_{0},\phi_{1},\psi_{0},\psi_{1})$ has compact support, we have
    \begin{equation*}
        \begin{aligned}
        \|\psi(t)\|^{2}+\|\partial\psi(t)\|^{2}
        &\lesssim \int_{\mathbb{S}^{2}}\int_{\R}
        \left(
        |\psi(t,r,\omega)|^{2}
        +|\partial \psi(t,r,\omega)|^{2}
        \right)r^{2}\d r\d \omega\\
        &\lesssim \epsilon^{2}\int_{0}^{t} \frac{r^{2}\ln^{4}\langle t-r\rangle}{\langle t+r\rangle^{2}\langle t-r\rangle^{2}}\d r
        \lesssim \epsilon^2\int_{0}^{t}\frac{1}{\langle t-r\rangle^{\frac{3}{2}}}\d r
        \lesssim 
        \epsilon^{2},
        \end{aligned}
    \end{equation*}
    which directly completes the proof for the estimates in (ii).
\end{proof}

Last, we study the growth of the natural energy $\|\partial \phi\|$ which, together with Lemma \ref{lem:growthofL2ofphi}, completes the proof of Theorem~\ref{thm:blow-up}. This relies on introducing a nonlinear transformation 
\begin{equation}\label{def:Xnon-transform}
X(\Phi, \Psi) = U\Phi - \ln v (\partial_t \Psi)^2.
\end{equation}

\begin{lemma}[Growth of $\|\partial \phi\|$]\label{prop:growthofprphi}
   Under the assumptions of Theorem~\ref{thm:blow-up},  it holds for sufficiently large $t$ that
    \begin{equation}\label{eq:growthofpartialphi}
        \mathfrak{c}_{5}\ln t\lesssim \|\partial \phi(t)\|\lesssim \epsilon \ln t,
    \end{equation}
    with the constant $\cfrak_5$ given as in \eqref{def:cfrak5constant}. Further, assume $\cfrak_1\neq 0$, then it holds $\cfrak_5>0$.
\end{lemma}
\begin{proof}
Without loss of generality, we assume that the initial data is posed on $\mathcal{K}_{1}=\mathcal{H}_{1}\cup\left\{(t,x):(t,r)\in \{\sqrt{5}\}\times [2,+\infty)\right\}$ and vanishes for $r\ge 1$ on $\mathcal{K}_{1}$. From now on, we consider the Cauchy problem~\eqref{equ:main} in the flat foliation of the spacetime, i.e., in constant-$t$ slices.

\smallskip
\textbf{Step 1}. Upper bound of $\|r^{-1} X(\Phi, \Psi)\|$.
First, from~\eqref{est:Firsttech} and the almost sharp decay~\eqref{est:almost1}, we deduce that 
\begin{equation*}
  |X(\Phi,\Psi)|\lesssim \epsilon u^{-1}\ln u\quad \mbox{on }\ \mathcal{D}_{\rm{ext},2\delta}\Longrightarrow 
  |X(\Phi,\Psi)|\lesssim \epsilon u^{-1}\ln u\quad \mbox{on }\  \mathcal{M}.
\end{equation*}
Integrating the above estimate over $r\in [0,t]$, we directly have 
\begin{align}\label{eq:X-upper}
    \|r^{-1} X(\Phi, \Psi)\|^2
    \lesssim \epsilon^{2} \int_{0}^{t}\frac{\ln^{2}\langle t-r\rangle}{\langle t-r\rangle^{2}}\d r
    \lesssim
    \epsilon^2.
\end{align}

\textbf{Step 2.} Lower bound of $\big\|r^{-1}(U \Psi)^2\big\|$.
Using again \eqref{equ:PhiPsi:waveUV}, we compute
\begin{equation*}
    V\left((U\Psi)^4\right) = 4 (U\Psi)^3 
    \left(r^{-2} \slashed{\Delta} \Psi + rQ_0(\phi, \phi) \right).
\end{equation*}
By the standard energy estimate in the space-time $[t,\infty)\times \R^{3}$ with the boundary $(\left\{t\right\}\times \R^{3})\cup \mathcal{I}$ and the fact that the initial data has compact support, we find 
\begin{equation*}
\begin{aligned}
   &\frac{1}{2}\int_{0}^{+\infty}\int_{\mathbb{S}^{2}}(U\Psi)^{4}(u,+\infty,\omega)\d \omega\d u -\|r^{-1}(U\Psi)^{2}\|^{2}\\
   &=4\int_{\mathbb{S}^{2}}\int_{t}^{\infty}\int_{0}^{t}(U\Psi)^{3}
    \left(r^{-2} \slashed{\Delta} \Psi + rQ_0(\phi, \phi) \right)\d r \d \sigma \d \omega.
   \end{aligned}
\end{equation*}
Using again the almost sharp decay~\eqref{est:almost1}, we obtain
\begin{equation*}
    |U\Psi|^{3} \left|r^{-2} \slashed{\Delta} \Psi + rQ_0(\phi, \phi) \right|
    \lesssim \epsilon^{4}v^{-2}u^{-6}(\ln v)(\ln^{6}u),
\end{equation*}
which, together with the definition of $\cfrak_5$ in \eqref{def:cfrak5constant}, directly implies 
\begin{equation}\label{eq:U-Psi-lower}
    \|r^{-1} (U\Psi)^2\|^2 \gtrsim (\mathfrak{c}_{5})^{2}+O(\epsilon^{4}t^{-1}\ln t).
\end{equation}

\textbf{Step 3.} Conclusion. 
First, from the almost sharp decay~\eqref{est:almost1}, we obtain 
\begin{equation*}
    \|\partial \phi\|^{2}\lesssim \|U\phi\|^{2}+\|V\phi\|^{2}
    \lesssim \ep^{2} \int_{0}^{t}\frac{\ln^{2}\langle t+r\rangle}{\langle t-r\rangle^{2}}\d r\lesssim \ep^{2}\ln^{2}t,
\end{equation*}
which completes the proof of the upper bound of $\|\partial \phi\|$.

\smallskip
On the other hand, from $2\partial_{t}=V+U$, we decompose
\begin{equation*}
    (\partial_{t}\Psi)^{4}=\frac{1}{16}(U\Psi)^{4}+O\left(|U\Psi|^{3}|V\Psi|+|V\Psi|^{4}\right).
\end{equation*}
Using again the almost sharp decay~\eqref{est:almost1}, we obtain
\begin{equation*}
    |U\Psi|^{3} |V\Psi|+|V\Psi|^{4}
    \lesssim \epsilon^{4}v^{-1}u^{-7}(\ln^{8}u),
\end{equation*}
which directly implies 
\begin{equation}\label{eq:VPSIUPPER}
  \left\|r^{-1}\left(|U\Psi|^{\frac{3}{2}}|V\Psi|^{\frac{1}{2}}+|V\Psi|^{2}\right)\right\|^{2}\lesssim \epsilon^{4}t^{-1}.
\end{equation}
Combining \eqref{eq:X-upper}--\eqref{eq:VPSIUPPER} with \eqref{def:Xnon-transform}, we obtain
\begin{equation*}
\begin{aligned}
  \|\partial \phi\|^{2}\ge  \frac{1}{2} \|r^{-1}U\Phi\|^{2}&\ge \frac{1}{4}\int_{\R^{3}}r^{-2}(\pt\Psi)^{4}\ln^{2}v\d x-\int_{\R^{3}}r^{-2}(X(\Phi,\Psi))^{2}\d x\\
    &\gtrsim (\mathfrak{c}_{5})^{2}\ln^{2}t
    +O\left(\ep^{4}t^{-1}\ln^{2}t\right)
+O\left(\ep^{4}t^{-1}\ln^{3}t\right)+O(\ep^{2}),
    \end{aligned}
\end{equation*}
which hence proves the estimate \eqref{eq:growthofpartialphi} for $t$ large enough.
In view of the definition of the constants $\cfrak_1$ and $\cfrak_5$ in \eqref{def:cfrak1and2:intro} and \eqref{def:cfrak5constant}, $\cfrak_1>0$ implies $\cfrak_5>0$. This concludes the proof of Lemma~\ref{prop:growthofprphi}.
\end{proof}

\end{document}